\documentclass[12pt, a4paper, reqno, captions=tableheading,bibliography=totoc]{scrartcl}

\usepackage[T1]{fontenc}
\usepackage{lmodern}
\usepackage[english]{babel}
\usepackage{amsfonts}
\usepackage{amsmath}
\usepackage{amssymb}
\usepackage{amsthm}
\usepackage[mathscr]{euscript}
\usepackage[dvipsnames,svgnames]{xcolor}
\usepackage{booktabs}
\usepackage{float}
\usepackage{enumitem}
\usepackage{graphicx}
\usepackage{tikz,pgf}
\usetikzlibrary{calc}
\usetikzlibrary{decorations.markings}
\usetikzlibrary{shapes.geometric}
\usetikzlibrary{patterns}
\usetikzlibrary{intersections}
\usetikzlibrary{arrows.meta}
\usetikzlibrary{cd}
\usepackage{array}
\usepackage{adjustbox}
\usepackage{calc}
\usepackage[hyphens]{url}
\usepackage{hyperref}
\hypersetup{
    colorlinks,
    linkcolor={red!50!black},
    citecolor={blue!50!black},
    urlcolor={blue!50!black}
}
\usepackage[nameinlink]{cleveref}
\usepackage{algorithm}
\usepackage{algpseudocode}

\theoremstyle{plain}
\newtheorem{lemma}{Lemma}[section]
\newtheorem{proposition}[lemma]{\textbf{Proposition}}
\newtheorem{theorem}[lemma]{\textbf{Theorem}}

\theoremstyle{definition}
\newtheorem{definition}[lemma]{\textbf{Definition}}
\newtheorem{example}[lemma]{\textbf{Example}}
\newtheorem*{notation}{\textbf{Notation}}
\newtheorem{remark}[lemma]{Remark}

\newcommand{\N}{\mathbb{N}}
\newcommand{\Z}{\mathbb{Z}}

\newcommand{\R}{\mathbb{R}}
\newcommand{\C}{\mathbb{C}}

\newcommand{\p}{\mathbb{P}}

\newcommand{\G}{G}
\newcommand{\calli}{H}

\algrenewcommand\algorithmicwhile{\textbf{While}}
\algrenewcommand\algorithmicfor{\textbf{For}}
\algrenewcommand\algorithmicdo{\textbf{Do}}
\algrenewcommand\algorithmicif{\textbf{If}}
\algrenewcommand\algorithmicthen{\textbf{Then}}
\algrenewcommand\algorithmicelse{\textbf{Else}}
\algrenewcommand\algorithmicend{\textbf{End}}
\algrenewcommand\algorithmicreturn{\textbf{Return}}

\newcommand{\M}[1]{\overline{\mathscr{M}}_{0,#1}}
\newcommand{\config}[2]{\mathscr{K}_{#1}^{#2}}
\newcommand{\redconfig}[2]{\widetilde{\mathscr{K}}_{#1}^{#2}}
\newcommand{\chow}[1]{K_{#1}}
\newcommand{\redchow}[1]{\widetilde{K}_{#1}}
\newcommand{\parameter}{\mathcal{U}}

\newcommand{\pcount}[1]{\mathbf{c}_2(#1)}
\newcommand{\scount}[1]{\mathbf{c}_{S^2}(#1)}

\newcommand{\lef}{L}
\newcommand{\righ}{R}
\newcommand{\cente}[1]{C_{#1}}

\newcommand{\quadform}[2]{\texttt{QuadForm(}{#1}\texttt{, }{#2}\texttt{)}}
\newcommand{\class}[1]{\texttt{S$^2$-Class(#1)}}

\newcommand*{\hexagon}{\foreach \i/\l/\s in {0/$P_2$/P2, 1/$P_0$/P0, 2/$P_1$/P1, 3/$Q_2$/Q2, 4/$Q_0$/Q0, 5/$Q_1$/Q1} \node[vertex, label={[labelsty]60*\i+30:\l}] (\s) at (60*\i+30:1) {};}

\colorlet{colG}{DarkSeaGreen}
\definecolor{colR}{HTML}{CC6677}
\definecolor{colO}{HTML}{DDCC77}
\definecolor{colB}{HTML}{6699CC}
\colorlet{colGray}{white!60!black}
\colorlet{colbg}{white}
\colorlet{colfg}{black}
\colorlet{colgraphv}{colfg!75!colbg}
\colorlet{colgraphe}{colfg!55!colbg}
\colorlet{colanchor}{colG} 
\colorlet{colvm}{colR} 
\colorlet{colvh}{colgraphv!60!colbg}
\colorlet{coleh}{colvh!40!colbg} 

\tikzstyle{vertex}=[fill=colgraphv,circle,inner sep=0pt, minimum size=4pt]
\tikzstyle{anchorvertex}=[vertex, colanchor]
\tikzstyle{mvertex}=[vertex, colvm]
\tikzstyle{vertexh}=[vertex, colvh]
\tikzstyle{edge}=[line width=1.5pt,colgraphe]
\tikzstyle{edgeh}=[edge,coleh]
\tikzstyle{labelsty}=[font=\scriptsize]

\tikzstyle{dedge}=[
	edge,%
	postaction={%
		decoration={markings,
			mark=at position #1
			with {
				\node[dart,fill,inner sep=1pt,minimum size=6pt,transform shape] (0,0) {};
			},
		},%
	decorate}%
]

\title{Calligraphs and sphere realizations}
\date{}

\author{%
Matteo Gallet$^{\diamond}$%
\and
Georg Grasegger$^{\circ}$%
\and
Niels Lubbes$^{\ast}$%
\and
Josef Schicho
}

\begin{document}
\maketitle
\footnotetext{\hspace{0.15cm}$^\diamond$ The project leading to this paper started when Matteo Gallet was supported by the Austrian Science Fund (FWF): P33003. At the moment, he is member of the Gruppo Nazionale Strutture Algebriche, Geometriche e le loro Applicazioni (GNSAGA) of the Istituto Nazionale di Alta Matematica (INdAM).\\%
$^\circ$ Supported by the Austrian Science Fund (FWF): P31888. This research was funded in part by the Austrian Science Fund (FWF): 10.55776/I6233.
	For open access purposes, the authors have applied a CC BY public copyright license to any author accepted manuscript version arising from this submission.\\
$^\ast$ Supported by the Austrian Science Fund (FWF): P33003 and~P36689.%
}

\begin{abstract}
  We introduce a recursive procedure for computing the number of realizations of a minimally rigid graph on the sphere up to rotations.
  We accomplish this by combining two ingredients.
  The first is a framework that allows us to think of such realizations as of elements of a moduli space of stable rational curves with marked points.
  The second is the idea of splitting a minimally rigid graph into two subgraphs, called calligraphs, that admit one degree of freedom
  and that share only a single edge and a further vertex.
  This idea has been recently employed for realizations of graphs in the plane up to isometries.
  The key result is that we can associate to a calligraph a triple of natural numbers with a special property:
  whenever a minimally rigid graph is split into two calligraphs,
  the number of realizations of the former equals the product of the two triples of the latter,
  where this product is specified by a fixed quadratic form.
  These triples and quadratic form codify the fact that
  we express realizations as intersections of two curves on the blowup of a sphere along two pairs of complex conjugate points.
\end{abstract}

\section{Introduction}
\label{introduction}

Given a graph without loops or multiedges, and a choice of length for each of its edges,
one can ask in how many ways one can realize the graph in the plane
so that the distance between two vertices connected by an edge coincides with the given edge length.
Clearly, once a realization compatible with the edge length assignment is found,
applying any isometry of the plane yields another realization compatible with the same edge length assignment.
Hence, it is reasonable to consider realizations up to isometries.
In this context, one family of graphs emerges as particularly interesting,
namely those graphs that admit only finitely many realizations, up to isometries, compatible with algebraically independent edge lengths.
These graphs are called \emph{minimally rigid}.
If we allow complex coordinates for the points of the realizations, then the number of realizations, up to isometries, of a minimally rigid graph~$G$ is the same, regardless of the choice of algebraically independent edges; it is denoted by~$\pcount{\G}$.
Similarly as in the plane, one can consider the problem of computing the number of realizations on the sphere up to direct isometries,
i.\,e., rotations.
In this case, the number is denoted by~$\scount{\G}$.
It turns out that graphs that are minimally rigid for the plane are also minimally rigid for the sphere;
see \cite{Pogorelov1973, Saliola2007, Izmestiev2009, EftekhariJacksonEtAl2018}.

Computing the number of realizations, up to isometries, of a minimally rigid graph,
or obtaining bounds for this number, is a question that has attracted recent interest,
as witnessed by several works;
see, for example, \cite{BorceaStreinu2004, SteffensTheobald2010, EmirisTsigaridasVarvitsiotis2010, CapcoGalletEtAl2018, JacksonOwen2019, GraseggerKoutschanTsigaridas2020, BartzosEmirisTzamos2021, BartzosEmirisEtAl2021, SadjadiHaghEtAl2021, BartzosEmirisTzamos2023, DewarGrasegger2024, Grasegger2025, DewarNixonSmith2026, SlyvkaRubesEtAl2026}.

\minisec{Our result}
In this paper, we combine the technique developed in~\cite{GraseggerElHilanyLubbes2024} with the framework from~\cite{GalletGraseggerSchicho2020} to provide new ways for counting complex realizations of minimally rigid graphs on the sphere up to rotations.

\minisec{Previous work}
The technique developed in~\cite{GraseggerElHilanyLubbes2024}, which can be applied to realizations in the plane,
splits a minimally rigid graph into two \emph{calligraphs} that share an edge and a further vertex.
A calligraph is a graph that, differently from a rigid one, admits a one-dimensional set of realizations up to isometries
when algebraically independent edge lengths are chosen; see \Cref{figure:calligraphs}.
\begin{figure}[ht]
  \centering
  \begin{tikzpicture}
    \begin{scope}
      \node[vertex] (1) at (0,0) {};
      \node[vertex] (2) at (2,0) {};
      \node[vertex] (3) at (2,2) {};
      \node[vertex] (4) at (0,2) {};
      \node[vertex] (0) at (1,4) {};
      \draw[edge] (1)edge(2) (2)edge(3) (3)edge(4) (4)edge(1) (3)edge(0) (4)edge(0);
    \end{scope}
    \begin{scope}[xshift = 4cm]
      \node[vertex] (1L) at (0,0) {};
      \node[vertex] (2L) at (0,2) {};
      \node[vertex] (3L) at (0,4) {};
      \node[vertex] (1R) at (2,0) {};
      \node[vertex] (2R) at (2,2) {};
      \node[vertex] (3R) at (2,4) {};
      \draw[edge] (1L)edge(1R) (1L)edge(2R) (1L)edge(3R) (2L)edge(1R) (2L)edge(2R) (2L)edge(3R) (3L)edge(1R) (3L)edge(2R);
    \end{scope}
  \end{tikzpicture}
  \caption{Examples of calligraphs.}
  \label{figure:calligraphs}
\end{figure}
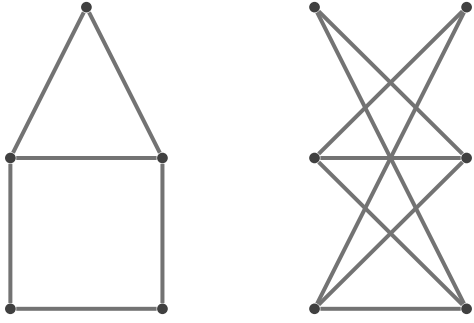
The realizations of the minimally rigid graph are then given by the intersections of the two curves, called \emph{coupler curves},
obtained by tracing the common vertex of the two calligraphs.
One could then think that it would be enough to understand the degree of the two coupler curves,
and to use B\'ezout's theorem to compute their number of intersections.
However, this turns out not to provide the correct number,
which means that we are in presence of intersections ``at infinity''.
The key observation is that,
even if we change the edge lengths of the minimally rigid graph,
the coupler curves of the two calligraphs may always intersect at two points at infinity.
Even more: these curves may pass through these two points with the same slope,
which in classical algebraic geometry language can be expressed by saying that
they further pass through four more points, which are ``infinitely near'' to the two initial ones.
Summing up, whatever minimally rigid graph we start from,
the coupler curves of two calligraphs that form a split may intersect at six fixed points,
regardless of the choice of edge lengths.

A standard technique in algebraic geometry is to change the space where the curves lie
in such a way that these unavoidable intersections are taken into account by the geometry of the space itself.
Employing this approach in the specific situation of calligraphs amounts to
moving from the plane to a particular surface (i.\,e., the blowup of~$\p^2$ at the fixed intersections)
and to use its intersection theory.
In the plane, as far as intersections are concerned, the only relevant information about a curve is its degree:
this is a way to re-state B\'ezout's theorem.
In other words, the intersection theory of the plane attaches to every curve a single integer
and to the intersection of two curves the product of these numbers.
In general, the geometry of a surface specifies
how to attach a tuple of integers to each curve lying on it
and a quadratic form to the whole surface:
the quadratic form specifies how to ``multiply'' two tuples in order to get the number of intersections of the two corresponding curves.

In the case of calligraphs, the authors of \cite{GraseggerElHilanyLubbes2024} attach a triple of integers, called the \emph{class}, to each calligraph.
The problem of computing the number~$\pcount{\G}$ of realizations in the plane up to isometries for a minimally rigid graph~$\G$ then reduces to computing the class~$(a,b,c)$ of any given calligraph~$\calli$:
in fact, once we split the graph~$\G$ into two calligraphs, it is enough to multiply the corresponding classes according to a specific quadratic form.
The solution to the problem of computing the class is to construct, out of the calligraph~$H$, three minimally rigid graphs~$\calli_1$, $\calli_2$, and~$\calli_3$.
Each of these three graphs splits into~$\calli$ and one of three ``basic calligraphs'' with at most four vertices, for which the classes are known;
see \Cref{figure:gadgets}.
Computing the number of realizations of $\calli_1$, $\calli_2$, and $\calli_3$ allows us to recover $(a,b,c)$.
This enables us to set up a recursive scheme for the computation of~$\pcount{\G}$.
\begin{figure}[ht]
  \centering
  \begin{tikzpicture}
    \begin{scope}
      \node[vertex] (1) at (0,0) {};
      \node[vertex] (2) at (2,0) {};
      \node[vertex] (0) at (2,2) {};
      \draw[edge] (1)edge(2) (1)edge(0);
      \node[] at (1,-1) {$\lef$};
    \end{scope}
    \begin{scope}[xshift=4cm]
      \node[vertex] (1) at (0,0) {};
      \node[vertex] (2) at (2,0) {};
      \node[vertex] (0) at (0,2) {};
      \draw[edge] (1)edge(2) (2)edge(0);
      \node[] at (1,-1) {$\righ$};
    \end{scope}
    \begin{scope}[xshift=8cm]
      \node[vertex] (1) at (0,0) {};
      \node[vertex] (2) at (2,0) {};
      \node[vertex] (0) at (1,2) {};
      \node[vertex] (4) at (1,1) {};
      \draw[edge] (1)edge(2) (1)edge(4) (2)edge(4) (4)edge(0);
      \node[] at (1,-1) {$\cente{v}$};
    \end{scope}
  \end{tikzpicture}
  \caption{The three ``basic calligraphs'', whose classes are known to be $(1,1,0)$ for $\lef$, $(1,0,1)$ for $\righ$, and $(2,0,0)$ for $\cente{v}$.}
  \label{figure:gadgets}
\end{figure}
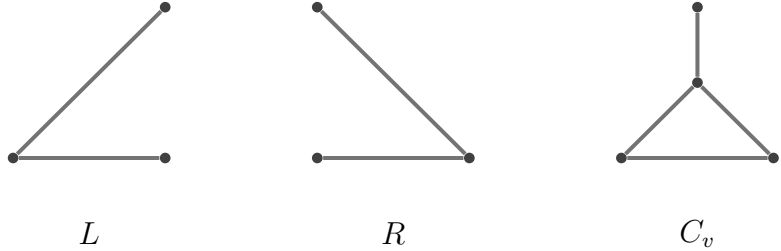

We see in this paper that, maybe surprisingly, the number~$\scount{\G}$ of realizations on the sphere up to rotations can be computed in the same way.
However, the proof technique of this result is fundamentally different from the one of~\cite{GraseggerElHilanyLubbes2024}
and, additionally, it avoids several technical issues that are present in the plane case.
This is due to the fact that, in the case of the sphere,
we have a well-behaved ambient space for the sets of realizations of graphs up to sphere rotations,
which are called \emph{configurations}.
In fact, the paper~\cite{GalletGraseggerSchicho2020} explains how we can think of configurations of a graph on the sphere
as elements of the so-called \emph{moduli space~$\M{2n}$ of rational stable curves with marked points},
in which two marked points correspond to one vertex of the graph and $n$ is the number of vertices.
Here, a \emph{rational stable curve} is a connected union of copies of the projective line that intersect each other transversely in at most one point.
The moduli space~$\M{2n}$ is a geometric object (more precisely, a smooth projective variety) whose elements correspond to equivalence classes of rational stable curves on which $2n$ points (the markings) are specified, up to the action of the automorphisms of the projective line.
The paper~\cite{GalletGraseggerSchicho2020} provides an algorithm to compute the number of configurations of a minimally rigid graph
by exploiting the structure of the Chow ring of the moduli space~$\M{2n}$.
The Chow ring of an algebraic variety is a ring whose elements, roughly speaking, are the subvarieties of the given variety and in which the product of two elements corresponds to the intersection of subvarieties.
In other words, the Chow ring encodes the intersection theory of an algebraic variety.
In our case, for each graph~$\G$ we can consider the subvariety~$\config{\G}{\Lambda}$ of configurations of~$\G$ with an assignment~$\Lambda$ of edge lengths.
This determines an element~$\chow{\G}$ in the Chow ring of~$\M{2n}$, which does not depend on~$\Lambda$.
When $\G$ is minimally rigid and $\Lambda$ is given by algebraically independent values, the subvariety~$\config{\G}{\Lambda}$ is a collection of points, one for each configuration of~$\G$.
The cardinality of~$\config{\G}{\Lambda}$ then equals~$\scount{G}$.
In turn, this cardinality is equal to the so-called \emph{degree} of~$\chow{G}$;
we can then perform the computation of~$\scount{G}$ by working only in the Chow ring,
since in addition we have that $\chow{\G}$ can be written as a product of particular elements corresponding to codimension one subvarieties (the so-called \emph{divisors}),
one for each of the edges of the graph.
This allows us to set up a recursive procedure for the computation of the degree of~$\chow{\G}$.

\minisec{The technique}
The goal of this paper is to enhance the algorithm just described to compute the number~$\scount{\G}$ by the same procedure of splitting a minimally rigid graph into calligraphs explained in the case of the plane.

A powerful feature of moduli spaces of stable curves with marked points is that
they come equipped with \emph{forgetful morphisms} of the form $\M{2n} \longrightarrow \M{2m}$ whenever $m \leq n$,
which forget a certain number of marked points of a rational stable curve.
Via these morphisms, we can relate the configuration space of a minimally rigid graph~$\G$
to those of two calligraphs~$\G_1$ and~$\G_2$ constituting a split of~$\G$
and to the one of the graph~$\G_{12}$ constituted of the common edge and vertex of~$\G_1$ and~$\G_2$.
All these morphisms are both proper and flat,
which means that they are most suited for operating between Chow rings.
Using them, we can express $\chow{\G}$
as a product of three contributions, coming from~$\G_1$, $\G_2$, and~$\G_{12}$
(these contributions are pullbacks, under suitable forgetful morphisms, of classes
determined by the corresponding graphs in the Chow rings of the appropriate moduli spaces).
By elementary computations in intersection theory,
we show that we can ``pushforward'' all these computations to the moduli space corresponding to~$\G_{12}$.
There, what we get is a product of two quantities depending on~$\G_1$ and~$\G_2$.
This product reduces to computing the intersection of two curves lying on the configuration space of~$\G_{12}$.
It turns out that the latter is isomorphic to the blow up of~$\p^1 \times \p^1$ at four points.
Therefore, in terms of intersection theory, each of the two curves is identified by six integers,
which actually become three because these curves have a real structure.
Hence, we are in the same setting as in the case of realizations in the plane described before
and so we can apply the same recursive strategy,
attaching to each calligraph a triple of integer numbers,
which we call the \emph{$S^2$-class}.

\minisec{Structure of the paper}
The paper is structured as follows.
In \Cref{minimally_and_calligraphs}, we recall the notions of minimally rigid graph, of calligraph, and of split of a minimally rigid graph as the union of two calligraphs that share a common edge and a common vertex.
At the end, we state the main result of this paper, \Cref{theorem:main}.
In \Cref{configurations}, we briefly summarize how to encode the realizations of a graph on the sphere up to rotations as a subvariety of a moduli space of rational stable curves with marked points.
\Cref{reduction} and \Cref{description} are the technical core of the paper.
\Cref{reduction} proves how it is possible to reduce the computation of the degree of the configuration space of a minimally rigid graph
to the degree of the intersection of three surfaces in a three-dimensional moduli space.
\Cref{description} provides an explicit description of one of these surfaces,
which turns out to be isomorphic to the blowup of~$\p^1 \times \p^1$ at four points.
\Cref{classes} combines the results from the previous two sections to define the $S^2$-class of a calligraph
and to show that the product of the two $S^2$-classes in a calligraphic split of a minimally rigid graph
gives the number of realizations of the latter on the sphere up to rotations, thus proving \Cref{theorem:main}.
\Cref{recursion} describes the recursive procedure that we obtain
and reports some proof-of-concept comparison between the timings of the original version of the algorithm from~\cite{GalletGraseggerSchicho2020}
against the ``enhanced'' version using the recursion of this paper.

\minisec{Acknowledgments.}
The authors would like to thank Paolo Aluffi, Valentina Beorchia, Renzo Cavalieri, Boulos El Hilany, Neithalath Mohan Kumar, Danilo Lewa\'nski, Andrea Ricolfi, and Elaine Wong for useful discussions.
They would like to thank especially Paolo Aluffi for pointing out to us the precise references that led to the proof of \Cref{lemma:commutativity_birational}.

\section{Minimally rigid graphs and calligraphs}
\label{minimally_and_calligraphs}

We start by introducing the concepts of \emph{minimally rigid graph} and of \emph{calligraph}.
The first one is a well-established and studied notion,
while the second one has recently been introduced in~\cite{GraseggerElHilanyLubbes2024}.
Once their edge lengths have been fixed,
minimally rigid graphs can be drawn in finitely many non-congruent ways in the plane,
while calligraphs allow a flex.
We define the notion of realization of a graph in the complex plane (and later on the complex sphere) and of their congruence.
This allows us to speak about the number of realizations of a minimally rigid graph up to isometries.
After defining what we mean by a \emph{split} of a minimally rigid graph into two calligraphs,
we state the main theorem of the paper (\Cref{theorem:main}), which is the theoretical bedrock of the recursive procedure we propose in \Cref{recursion}.
Notice that, in this paper, all considered graphs are \emph{simple}, namely without loops or multiedges.

A \emph{(complex) realization} of a graph $\G = (V, E)$ is a function $\rho \colon V \longrightarrow \C^2$.
Given a choice of \emph{edge lengths} $\lambda = (\lambda_e)_{e \in E}$ with $\lambda_e \in \R_{>0}$ for all $e \in E$,
we say that a realization~$\rho$ \emph{induces}~$\lambda$ if
\[
  \bigl\langle
    \rho(u) - \rho(v),
    \rho(u) - \rho(v)
  \bigr\rangle
  =
  \lambda_e
\]
for all edges $e = \{u,v\}$ in~$E$,
where $\left\langle \cdot, \cdot \right\rangle$ is the quadratic form associated to the identity matrix.
Notice that, although realizations of vertices can have complex coordinates,
we always ask edge lengths to be non-zero \emph{real} numbers.
This is of crucial importance for the development of our results,
because from this assumption it follows that all the objects that we are going to define
are endowed with a real structure.
Two realizations are considered \emph{congruent} up to (complex) isometries
if there exists a matrix $A \in \mathrm{SO}_2(\C)$ and a vector $b \in \C^2$
such that the map $v \mapsto Av + b$ sends one realization to the other.

\begin{definition}
  \label{definition:minimally_rigid}
  A graph~$\G = (V, E)$ is called \emph{minimally rigid} if
  it admits only finitely many, but not zero, (complex) realizations in the plane
  for a general choice of edge lengths, up to (complex) isometries.
  This number of realizations is denoted by~$\pcount{\G}$.
\end{definition}

A theorem of Pollaczek-Geiringer \cite{Geiringer1927} and Laman \cite{Laman1970} characterizes minimally rigid graphs in combinatorial terms.

\begin{theorem}[Pollaczek-Geiringer, Laman]
  \label{theorem:characterization_minimally}
  A graph~$\G = (V, E)$ is \emph{minimally rigid} if and only if
  \begin{itemize}
    \item $|E| = 2|V| - 3$ and
    \item $|E'| \leq 2|V'| - 3$ for every subgraph $\G' = (V', E')$ with at least two vertices.
  \end{itemize}
\end{theorem}

We can adapt all notions we introduced so far in order to define realizations of graphs on the (complex) sphere
\[
  S^2_{\C} :=
  \bigl\{
    (x,y,z) \in \C^3
    \, \mid \,
    x^2 + y^2 + z^2 = 1
  \bigr\} \,.
\]
Hence, we can speak of minimally rigid graphs on the sphere and if $\G$ is one of them,
one denotes by~$\scount{\G}$ the number of its (complex) realizations on~$S^2_{\C}$ up to (complex) rotations.
One then finds out that a graph is minimally rigid on the sphere if and only if it is so in the plane;
see \cite{Pogorelov1973, Saliola2007, Izmestiev2009, EftekhariJacksonEtAl2018}.

The paper~\cite{GalletGraseggerSchicho2020} provides a recursive procedure for the computation of~$\scount{\G}$
and the goal of this paper is to set up a technique that provides sometimes very sensible speed ups,
mimicking the one of~\cite{GraseggerElHilanyLubbes2024}.
To do so, we introduce another type of graphs, called \emph{calligraphs}.
These graphs, roughly speaking, admit a one-dimensional set of realizations up to isometries
and are the key players in the recursive procedure.

\begin{definition}
  \label{definition:calligraph}
  A \emph{calligraph} is a graph~$\calli=(V,E)$ together with a choice of an edge~$\bar{e} = \{ \bar{u}_1, \bar{u}_2 \}$ and of a vertex~$\bar{v}$
  such that
  \begin{itemize}
    \item $|E| = 2|V| - 4$ and
    \item $(V, E \cup \{ \{\bar{v}, \bar{u}_1\}\})$ is rigid or $(V, E \cup \{ \{\bar{v}, \bar{u}_2\}\})$ is rigid.
  \end{itemize}
\end{definition}

\begin{notation}
  Throughout this paper, if $(\calli, \bar{e}, \bar{v})$ is a calligraph,
  we always suppose that $\bar{e} = \{1, 2\}$ and $\bar{v} = 0$.
  Hence, from now on we always refer to a calligraph simply by its underlying graph.
\end{notation}

The key point of the recursive technique that we establish is to split a minimally rigid graph into a union of two calligraphs with a common edge and a further common vertex.

\begin{definition}
  \label{definition:calligraphic_split}
  A \emph{split} of a graph~$\G = (V, E)$ is a pair~$(\G_1, \G_2)$ of subgraphs of~$\G$
  such that if we write $\G_i = (V_i, E_i)$, then
  \begin{itemize}
    \item $V = V_1 \cup V_2$ and $E = E_1 \cup E_2$;
    \item $V_1 \cap V_2 = \{0,1,2\}$ and $E_1 \cap E_2 = \{ \{ 1, 2 \} \}$.
  \end{itemize}
  If both $\G_1$ and $\G_2$ are calligraphs, then $(\G_1, \G_2)$ is called a \emph{calligraphic split}.
\end{definition}

As already mentioned in \Cref{introduction}, three among all calligraphs are going to play a special role in our recursion.
\begin{definition}
  \label{definition:basic_calligraphs}
  We define three ``basic calligraphs'', namely
  \begin{itemize}
    \item $\lef := \bigl( \{0, 1, 2\}, \{\{1,2\},\{0,1\}\} \bigr)$,
    \item $\righ := \bigl( \{0, 1, 2\}, \{\{1,2\},\{0,2\}\} \bigr)$, and
    \item for $v \in \N$ with $v > 2$, we set $\cente{v} := \bigl( \{0, 1, 2, v\}, \{\{1,2\},\{0,v\},\{1,v\},\{2,v\}\} \bigr)$.
  \end{itemize}
  These are the three graphs from \Cref{figure:calligraphs}.
\end{definition}

We conclude this section by stating the main theorem of the paper.
This result shows that we can attach to each calligraph a triple of integer numbers
so that, once we have a calligraphic split of a minimally rigid graph~$\G$,
the number~$\scount{\G}$ can be computed from the two triples of the split.

\begin{theorem}
  \label{theorem:main}
  There exist a function
  \[
    [\,\_\,] \colon \{ \text{calligraphs} \} \longrightarrow \Z^3
  \]
  and a quadratic form
  \[
  \begin{array}{rccc}
    \_ \cdot \_ \colon & \Z^3 \times \Z^3 & \longrightarrow & \Z \\
    & (a_1, b_1, c_1), (a_2, b_2, c_3) & \mapsto & 2 (a_1 a_2  - b_1 b_2 - c_1 c_2)
  \end{array}
  \]
  such that,
  if $\G$ is a minimally rigid graph and $(\G_1, \G_2$) is a calligraphic split of~$\G$,
  then the number of complex realizations of~$\G$ on the sphere up to complex rotations, i.\,e., the quantity~$\scount{\G}$, is $[\G_1] \cdot [\G_2]$. The triple~$[\calli]$ is called the \emph{$S^2$-class} of~$\calli$.

  Moreover, the $S^2$-classes of the three ``basic calligraphs'' from \Cref{definition:basic_calligraphs} are
  \[
    [\lef] = (1,1,0)\,, \quad
    [\righ] = (1,0,1)\,, \quad
    [\cente{v}] = (2,0,0)\,.
  \]
\end{theorem}

\begin{remark}
  \label{remark:644}
  As explained in \Cref{classes}, one could also choose $(6,4,4)$ for the $S^2$-class of~$\cente{v}$
  and the rest of \Cref{theorem:main} would stay the same.
\end{remark}

\Cref{theorem:main} is the spherical analogue of \cite[Theorem I]{GraseggerElHilanyLubbes2024}.
However, the proof strategy that we employ here is completely different from the one of~\cite{GraseggerElHilanyLubbes2024}.
We believe it would be interesting to have a ``moduli-based'' proof of this result also in the case of the plane,
but this remains an open problem.

\section{Graph sphere configurations in moduli spaces}
\label{configurations}

In this section, we briefly recall the construction of \cite{GalletGraseggerSchicho2020}
that assigns to each graph~$\G$ with $n$ vertices together with a choice of edge lengths
a subvariety of the moduli space of rational stable curves with $2n$ marked points.
This subvariety is called the \emph{configuration space} of the graph
and it is defined in such a way that
an open subset of it is in bijection with the congruence classes of realizations of~$\G$ on the sphere inducing the given edge lengths up to (complex) isometries.
The idea is that:
\begin{itemize}
  \item the projective closure of the complex sphere~$S^2_{\C}$ intersects the plane at infinity in a smooth conic, called the \emph{absolute conic};
  \item through each point on the complex sphere~$S^2_{\C}$ pass two lines, each of which intersects the absolute conic in a point.
\end{itemize}
Hence, to a realization of $n$ points on the sphere, we can associate $2n$ points on the absolute conic.
Then, one notices that two realizations of $n$ points are (complex) rotation-congruent if and only if
the corresponding two $2n$ tuples of points on the absolute conic differ by a projective automorphism.
Thus, realizations of $n$ vertices up to (complex) rotations can be encoded as elements of the moduli space of rational stable curves with $2n$ markings,
which is usually denoted by~$\M{2n}$.
Moreover, the cosine of the angle between two points on the sphere can be expressed
as the cross-ratio between the four corresponding points on the absolute conic.
For the precise details, we refer to \cite{GalletGraseggerSchicho2020} and the references therein.

Notice that the real projective automorphisms of $S^2_{\C}$ that preserve the
absolute conic correspond to exactly rotations of $S^2_{\R}$.
Moreover, the real structure of~$S^2_{\C}$ given by componentwise complex conjugation
swaps the two lines passing through each point of~$S^2_{\R}$.
This induces a real structure on the moduli space of rational stable curves with $2n$ markings that swaps two $n$-tuples of markings.

\begin{definition}
  \label{definition:moduli_vertices}
  For a set $V$, we write $\M{2V}$ to denote the moduli space~$\M{2|V|}$ of rational stable curves with $2|V|$ marked points,
  where the marked points are indicated by $\{ P_v \}_{v \in V} \cup \{ Q_v \}_{v \in V}$.
  The moduli space~$\M{2V}$ is a smooth projective variety.
\end{definition}

With this notation, the real structure on $\M{2V}$ described before swaps the $P$-marked points with the $Q$-marked points.

To encode the computation of the cross-ratio in the framework of moduli spaces,
we introduce a particular type of morphisms.

\begin{definition}
  \label{definition:forgetful_morphisms}
  Let $\G = (V, E)$ be a graph and let $e \in E$ with $e = \{ v, w \}$.
  We define $\pi_e$ to be the \emph{edge forgetful morphism}
  \[
    \pi_e \colon \M{2V} \longrightarrow \M{2e}
  \]
  that forgets all the marked points except for $P_v$, $Q_v$, $P_w$, and $Q_w$.
\end{definition}

By unraveling the definitions,
one sees that the edge forgetful morphisms compute the cross-ratio of the four points that are not forgotten,
when we identify~$\M{2e}$ with~$\p^1$.
Hence, in our perspective, fixing the value of an edge forgetful morphism
amounts to fixing the spherical length of an edge.
Therefore, if we are given an assignment of edge lengths $\lambda = (\lambda_e)_{e \in E}$ for a graph $\G = (V, E)$,
then we can think of it as an element $\Lambda \in \prod_{e \in E} \M{2e}$
and we can look after all elements in the moduli space~$\M{2V}$ that are mapped to~$\Lambda$ by the edge forgetful morphisms.
These elements constitute the configurations of the graph,
namely a compactification of the set of realizations up to rotations.

\begin{definition}
  \label{definition:configuration_space}
  Let $\G = (V,E)$ be a graph and let $\Lambda \in \prod_{e \in E} \M{2e}$.
  The \emph{configuration space}~$\config{\G}{\Lambda}$ of~$\G$ with labeling~$\Lambda$ is defined as
  \[
    \config{\G}{\Lambda} := \Phi_{\G}^{-1}(\Lambda) \,,
  \]
  where
  \[
    \Phi_\G \colon \M{2V} \longrightarrow \prod_{e \in E} \M{2e}
  \]
  is the product $\prod_{e \in E} \pi_e$.
  If the element~$\Lambda$ comes from a real assignment of edge lengths,
  then the variety~$\config{\G}{\Lambda}$ is real with respect to the real structure
  that swaps the $P$-marked points and the $Q$-marked points.
  We denote by $\chow{\G}^{\Lambda}$ the rational equivalence class of~$\config{\G}{\Lambda}$ in the Chow ring of~$\M{2V}$.
  Since this class does not depend on~$\Lambda$, from now on we simply write $\chow{\G}$.
\end{definition}

For reasons that are made clear in \Cref{reduction},
once we are given a split of a graph,
it is convenient to remove the constraint given by the edge common to the two calligraphs
and define two new configuration spaces.

\begin{definition}
  \label{definition:reduced_configuration_space}
  Let $\G = (V, E)$ and let $(\G_1, \G_2)$ be a split of~$\G$.
  Let $\Lambda \in \prod_{e \in E} \M{2e}$ be an edge labeling for~$\G$.
  For $i \in \{1, 2\}$, we set
  \[
    \widetilde{E}_i := E_i \setminus \{ \{1,2\} \} \quad \text{and} \quad \widetilde{\Lambda}_i := (\Lambda_e)_{e \in \widetilde{E}_i} \,.
  \]
  Moreover, we define
  \[
    \widetilde{\Phi}_{\G_i} \colon \M{2V} \longrightarrow \prod_{e \in \widetilde{E}_i} \M{2e}
  \]
  as the product $\prod_{e \in \widetilde{E}_i} \pi_e$.
  We define the \emph{reduced configuration space}~$\redconfig{\G_i}{\widetilde{\Lambda}_i}$ of~$\G_i$ to be
  \[
    \redconfig{\G_i}{\widetilde{\Lambda}_i} := \widetilde{\Phi}_{\G_i}^{-1}(\widetilde{\Lambda}_i) \,.
  \]
  As in \Cref{definition:configuration_space}, if $\widetilde{\Lambda}_i$ comes from a real edge length assignment,
  then $\redconfig{\G_i}{\widetilde{\Lambda}_i}$ is a real subvariety of~$\M{2V}$.
  We denote by $\redchow{\G_i}$ the rational equivalence class of~$\redconfig{\G_i}{\widetilde{\Lambda}_i}$ in the Chow ring,
  which does not depend on $\widetilde{\Lambda}_i$ either.
\end{definition}

When a graph is minimally rigid, its configuration space for a general choice of edge lengths is a finite set,
whose cardinality equals the number of realizations of the graph on the sphere up to rotations.
Therefore, this number is equal to the degree of the class corresponding to the configuration space of the graph,
as stated in \Cref{proposition:count_degree}.

\begin{proposition}[\cite{GalletGraseggerSchicho2020}]
  \label{proposition:count_degree}
  Let $\G = (V, E)$ be a minimally rigid graph. Then
  \[
    \scount{\G} = \int_{\M{2V}} K_\G \,.
  \]
  Here and in the following, $\displaystyle \int_X D$ indicates the \emph{degree} of a class $D$ in the Chow ring of a complete variety~$X$ over $\C$,
  which is the number $s_{\ast}(D)$, where $s \colon X \longrightarrow \mathrm{Spec}(\C)$ is the structural morphism of~$X$
  and one identifies the Chow group of classes of dimension zero in~$\mathrm{Spec}(\C)$ with $\Z$; see \cite[Definition~1.4]{Fulton1998}.
  When $D$ is the class of a finite reduced scheme, its degree is the number of closed points of that scheme.
\end{proposition}

\begin{remark}
 The quantity $\displaystyle \int_{\M{2V}} K_\G$ is a special case of a \emph{cross-ratio degree}, see \cite{Silversmith2022,Silversmith2024}, and of a \emph{Kapranov degree}, see \cite[Definition~1.2]{BrakensiekEurEtAl2025}.
\end{remark}

\section{Reduction to three classes in \texorpdfstring{$\M{6}$}{M06}}
\label{reduction}

The goal of this section is to show that it is possible to express the degree of $\chow{\G}$
as an intersection of three contributions in the Chow ring of~$\M{6}$ (see \Cref{proposition:reduction}).
From now on, in this section we fix the following notation:
\begin{itemize}
  \item $\G = (V, E)$ is a graph;
  \item $(\G_1, \G_2)$ is a split of~$\G$ and $\G_i = (V_i, E_i)$ for $i \in \{1,2\}$;
  \item $\G_{12}$ is the graph with vertices $V_{12} := \{ 0, 1, 2 \}$ and edges $E_{12} := \{ \{1, 2\} \}$.
\end{itemize}
From this setting, we get the following commutative diagram, where all the maps are forgetful morphisms
(hence, in particular, all maps are flat and proper):
\begin{equation}
  \label{equation:main_diagram}
  \begin{tikzcd}
    & \M{2V} \arrow[ld, "\sigma_1"'] \arrow[rd, "\sigma_2"] \arrow[dd, "\vartheta"] \\
    \M{2V_1} \arrow[rd, "\eta_1"'] & & \M{2V_2} \arrow[ld, "\eta_2"] \\
    & \M{2V_{12}}
  \end{tikzcd}
\end{equation}
First of all, we write $\chow{\G}$ as the product of three contributions in~$\M{2V}$,
two of which are the pullbacks of the classes of the reduced configuration spaces from \Cref{definition:reduced_configuration_space}.

\begin{lemma}
  \label{lemma:inspection}
  $\chow{\G} = \vartheta^{\ast}(\chow{\G_{12}}) \cdot \sigma_1^{\ast}(\redchow{\G_1}) \cdot \sigma_2^{\ast}(\redchow{\G_2})$ \,.
\end{lemma}
\begin{proof}
  The statement follows from inspecting the description of~$\chow{\G}$ provided in~\cite{GalletGraseggerSchicho2020}, which we now recall.

  We can write $\chow{\G}$ as a product of the classes of particular divisors in~$\M{2V}$, denoted by~$D_{I, J}$, where $(I, J)$ is a partition of~$2V$:
  \[
    \chow{\G} = \prod_{\substack{e \in E \\ e = \{ v, w \}}} \sum_{\substack{(I,J) \text{ s.t.} \\ P_v, Q_v \in I \\ P_w, Q_w \in J}} D_{I,J} \,.
  \]
  The precise description of the classes $D_{I, J}$ is not needed in this paper;
  for the interested reader, we refer to~\cite{GalletGraseggerSchicho2020} and the references therein.

  With this description at hand, we decompose $\chow{G}$ into three contributions:
  \begin{multline*}
    \chow{\G} = \prod_{\substack{e \in E \\ e = \{ v, w \}}} \sum_{\substack{(I,J) \text{ s.t.} \\ P_v, Q_v \in I \\ P_w, Q_w \in J}} D_{I,J} = \\
    \underbrace{\prod_{\substack{e \in E_{12} \\ e = \{ v, w \}}} \sum_{\substack{(I,J) \text{ s.t.} \\ P_v, Q_v \in I \\ P_w, Q_w \in J}} D_{I,J}}_{\vartheta^{\ast}(\chow{\G_{12}})} \, \cdot \,
    \underbrace{\prod_{\substack{e \in \widetilde{E}_1 \\ e = \{ v, w \}}} \sum_{\substack{(I,J) \text{ s.t.} \\ P_v, Q_v \in I \\ P_w, Q_w \in J}} D_{I,J}}_{\sigma_1^{\ast}(\redchow{\G_1})} \, \cdot \,
    \underbrace{\prod_{\substack{e \in \widetilde{E}_2 \\ e = \{ v, w \}}} \sum_{\substack{(I,J) \text{ s.t.} \\ P_v, Q_v \in I \\ P_w, Q_w \in J}} D_{I,J}}_{\sigma_2^{\ast}(\redchow{\G_2})}
  \end{multline*}
  where the $\widetilde{E}_i$'s are as in \Cref{definition:reduced_configuration_space}.
\end{proof}

The main result that we prove in this section is the following.

\begin{proposition}
  \label{proposition:reduction}
  $\displaystyle \int_{\M{2V}} K_\G = \int_{\M{2V_{12}}} \chow{\G_{12}} \cdot (\eta_1)_{\ast} (\redchow{\G_1}) \cdot (\eta_2)_{\ast} (\redchow{\G_2})$\,.
\end{proposition}

\Cref{proposition:reduction} results immediately from the following, stronger, result.

\begin{proposition}
  \label{proposition:reduction_stronger}
  $\vartheta_{\ast} (K_\G) =  \chow{\G_{12}} \cdot (\eta_1)_{\ast} (\redchow{\G_1}) \cdot (\eta_2)_{\ast} (\redchow{\G_2})$ \,.
\end{proposition}

The rest of the section is devoted to proving \Cref{proposition:reduction_stronger}.

First of all, using the so-called \emph{projection formula} (see \cite[Proposition 8.3(c)]{Fulton1998}) and \Cref{lemma:inspection}, we get:
\[
  \vartheta_{\ast} (K_\G) =
  \vartheta_{\ast}
  \bigl(
    \vartheta^{\ast}(\chow{\G_{12}}) \cdot \sigma_1^{\ast}(\redchow{\G_1}) \cdot \sigma_2^{\ast}(\redchow{\G_2})
  \bigr) =
  \chow{\G_{12}} \cdot \vartheta_{\ast}
  \bigl(
    \sigma_1^{\ast}(\redchow{\G_1}) \cdot \sigma_2^{\ast}(\redchow{\G_2})
  \bigr) \,.
\]
Hence, to prove \Cref{proposition:reduction_stronger}, it is enough to show that
\begin{equation}
  \label{equation:equality}
  \vartheta_{\ast}
  \bigl(
    \sigma_1^{\ast}(\redchow{\G_1}) \cdot \sigma_2^{\ast}(\redchow{\G_2})
  \bigr) =
  (\eta_1)_{\ast} (\redchow{\G_1}) \cdot (\eta_2)_{\ast} (\redchow{\G_2}) \,.
\end{equation}

\Cref{equation:equality} is implied by the following claim.

{\textbf{Claim.}} $(\sigma_2)_{\ast} (\sigma_1)^{\ast} = (\eta_2)^{\ast} (\eta_1)_{\ast}$\,.

In fact, if the claim holds we have that
\begin{align*}
  \vartheta_{\ast} \bigl( \sigma_1^{\ast}(\redchow{\G_1}) \cdot \sigma_2^{\ast}(\redchow{\G_2}) \bigr) &= (\eta_2 \, \sigma_2)_{\ast} \bigl( \sigma_1^{\ast}(\redchow{\G_1}) \cdot \sigma_2^{\ast}(\redchow{\G_2}) \bigr) \\
  &= (\eta_2)_{\ast} \bigl(  (\sigma_2)_{\ast} \sigma_1^{\ast}(\redchow{\G_1}) \cdot \redchow{\G_2} \bigr) \quad \text{(proj.\ formula)} \\
  &= (\eta_2)_{\ast} \bigl( \eta_2^{\ast} (\eta_1)_{\ast}(\redchow{\G_1}) \cdot \redchow{\G_2} \bigr)  \quad \text{(claim)} \\
  &= (\eta_1)_{\ast}(\redchow{\G_1}) \cdot(\eta_2)_{\ast}(\redchow{\G_2})  \quad \text{(proj.\ formula)} \,.
\end{align*}
Thus, we are left with proving the claim.
Notice that if the diagram in \Cref{equation:main_diagram} was a fiber product, then the claim would immediately follow from the commutativity of flat pullbacks with proper pushforwards (see \cite[Proposition~1.7]{Fulton1998}) since all maps are both flat and proper.
However, this is not the case.
Still, we can prove that the diagram is ``not too far'' from a fiber square, in the following sense.
Consider
\begin{equation}
  \label{equation:fiber_product}
  \begin{tikzcd}
    & \M{2V} \arrow[ldd, "\sigma_1"', bend right] \arrow[rdd, "\sigma_2", bend left] \arrow[d, "\gamma"] \\
    & \M{2V_1} \times_{\M{2V_{12}}} \M{2V_2} \arrow[ld, "\alpha_1"] \arrow[rd, "\alpha_2"'] \\
    \M{2V_1} \arrow[rd, "\eta_1"'] & & \M{2V_2} \arrow[ld, "\eta_2"] \\
    & \M{2V_{12}}
  \end{tikzcd}
\end{equation}
where $\gamma$ is the morphism determined by the universal property of the fiber product.

\begin{lemma}
  \label{lemma:proper_birational}
  The map $\gamma$ is proper and birational.
\end{lemma}
\begin{proof}
  The morphism~$\gamma$ is proper since both $\M{2V}$ and $\M{2V_1} \times_{\M{2V_{12}}} \M{2V_2}$ are projective varieties.
  We now show that it is birational.
  For this, recall that, from the general construction of $\M{2V}$, there exists an open subset~$\mathscr{M}_{2V}$ of~$\M{2V}$ whose elements are irreducible stable curves with marked points,
  namely, tuples of $2|V|$ distinct points in~$\p^1$.
  We restrict $\gamma$ to~$\mathscr{M}_{2V}$ and prove that it is an isomorphism.
  The map $\gamma$ sends a tuple $(P_v)_{v \in V} \cup (Q_v)_{v \in V}$ to the pair of tuples $(P_v)_{v \in V_1} \cup (Q_v)_{v \in V_1}$ and $(P_v)_{v \in V_2} \cup (Q_v)_{v \in V_2}$, which share the tuples $(P_v)_{v \in V_{12}} \cup (Q_v)_{v \in V_{12}}$, hence $\gamma$ is an isomorphism between~$\mathscr{M}_{2V}$ and its image.
\end{proof}
Now that we know that \Cref{equation:main_diagram} ``differs'' from a fiber square by a proper birational morphism,
we can use the following result to ``transfer'' the commutativity of maps from the fiber square to the outer square in \Cref{equation:fiber_product}.

We do not claim any originality about \Cref{lemma:commutativity_birational}, see the Acknowledgments.
Any errors in adapting the explained proof should be ascribed only to us.

\begin{lemma}
  \label{lemma:commutativity_birational}
  Let $f \colon X \longrightarrow Y$ and $g \colon Y \longrightarrow Z$ be morphisms of varieties, let $h := g \circ f$.
  Suppose that:
  \begin{itemize}
    \item $f$ is proper and birational;
    \item $g$ and $h$ are flat;
    \item $Z$ is non-singular.
  \end{itemize}
  Then $f_{\ast} \, h^{\ast} = g^{\ast}$.
\end{lemma}
\begin{proof}
  This result follows from the combination of \cite[Proposition 8.1.1(c)]{Fulton1998} and \cite[Proposition 8.1.2(a)]{Fulton1998}.
  In fact, in \cite[Proposition 8.1.1(c)]{Fulton1998} we take:
  \begin{itemize}
    \item $f$, $g$, and $h$ as in our statement;
    \item $X'=X$, $Y' = Y$, and $Z' = Z$, furthermore $p_X = \mathrm{id}_X$, $p_Y = \mathrm{id}_Y$, and $p_Z = \mathrm{id}_Z$;
    \item $x = [X]$.
  \end{itemize}
  With this choice, the following formula holds:
  \[
    f_{\ast}( [X] \cdot_{h} z) = f_{\ast}([X]) \cdot_{g} z
  \]
  for any class~$z$ in~$Z$, as an equality in the Chow group of~$Y$.
  Since $f$ is birational, we have that $f_{\ast}([X]) = [Y]$.
  Both $g$ and $h$ are flat, so \cite[Proposition 8.1.2(a)]{Fulton1998} gives us that
  \[
    [X] \cdot_{h} z = h^{\ast}(z)
    \quad \text{and} \quad
    [Y] \cdot_{g} z = g^{\ast}(z) \,.
  \]
  Therefore, altogether the previous equalities tell us that for any class~$z$ in~$Z$, we have
  \[
    f_{\ast} \, h^{\ast} (z) = g^{\ast}(z) \,,
  \]
  which is the desired result.
\end{proof}

We have all the tools to prove the claim and conclude this section.
Notice that both the maps~$\alpha_1$ and~$\alpha_2$ are proper and flat since $\eta_1$ and~$\eta_2$ are so.
Hence, \Cref{lemma:proper_birational} allows us to apply \Cref{lemma:commutativity_birational} with $f = \gamma$, $g = \alpha_1$, and $h = \sigma_1$ to get
\[
  \gamma_{\ast} \sigma_1^{\ast} = \alpha_1^{\ast} \,.
\]
Since flat pullbacks commute with proper pushforwards in a fiber square, we have $(\alpha_2)_{\ast} (\alpha_1)^{\ast} = (\eta_2)^{\ast} (\eta_1)_{\ast}$. Then
\begin{align*}
  (\eta_2)^{\ast} (\eta_1)_{\ast} &= (\alpha_2)_{\ast} (\alpha_1)^{\ast} \\
  &= (\alpha_2)_{\ast} \gamma_{\ast} (\sigma_1)^{\ast} \\
  &= (\sigma_2)_{\ast} (\sigma_1)^{\ast} \,.
\end{align*}
This concludes the proof of the claim and hence of \Cref{proposition:reduction_stronger}.
Therefore, we reduced the computation of the class of the configuration space of a graph~$\G$
to the intersection of three classes in the $3$-fold $\M{6}$,
one of which does not depend on~$\G$.
Notice that so far we have neither supposed that $\G$ is minimally rigid, nor that $\G_1$ and~$\G_2$ are calligraphs.

\section{Description of the configurations of the fixed edge}
\label{description}

Recall from \Cref{reduction} that the graph~$\G_{12}$ has vertices $V_{12} = \{0,1,2\}$ and edges $E_{12} = \{ \{1,2\} \}$.
Therefore, the map $\Phi_{\G_{12}}$ as in \Cref{definition:configuration_space} has $\M{2V_{12}} \cong \M{6}$ as domain and $\M{2\{1,2\}} \cong \M{4} \cong \p^1$ as codomain.
Moreover, recall again from \Cref{definition:configuration_space} that $\config{\G_{12}}{\Lambda_{12}}$ is the fiber of~$\Phi_{\G_{12}}$ over an element~$\Lambda_{12} \in \M{2\{1,2\}}$.
The goal of this section is to prove that, for a general $\Lambda_{12}$,
the variety~$\config{\G_{12}}{\Lambda_{12}}$ is isomorphic to the blowup of~$\p^1 \times \p^1$ at four points (\Cref{proposition:description}).
This may be well-known to experts, but we could not locate this result in the literature;
moreover, we believe that the proof that we provide may be accessible with some acquaintance about moduli spaces and intersection theory.

Let us start by remarking that the moduli space~$\M{2V_{12}}$ of stable rational curves with six marked points,
where $\config{\G_{12}}{\Lambda_{12}}$ lies,
is isomorphic to the blowup of the Segre cubic~$S^3$ threefold along its $10$ nodes (see \cite[Section~3]{BeckmannBelmans2022}).
The Segre cubic threefold~$S_3$ is the unique (up to projective transformations) cubic hypersurface in~$\p^4$ with $10$ nodes (see for example \cite{Dolgachev2016}).
It admits an interpretation as a moduli space for $6$-tuples of points in~$\p^1$ with at most three coinciding points
that we are soon going to recall.

Recalling from \Cref{definition:configuration_space} the description of the map~$\Phi_{\G_{12}}$, we have a commutative diagram
\begin{equation}
  \label{equation:resolution_segre}
  \begin{tikzcd}[column sep=huge]
    \M{2V_{12}} \arrow[r, "\text{blowup}"] \arrow[d, "\Phi_{\G_{12}}"'] &  S_3 \arrow[ld, "\Psi_{\G_{12}}", dashrightarrow] \\
    \M{2\{1,2\}}
  \end{tikzcd}
\end{equation}
that defines the rational map~$\Psi_{\G_{12}}$.
We know that $\config{\G_{12}}{\Lambda_{12}}$ is a fiber of~$\Phi_{\G_{12}}$ over a general point.
We compute the (closure of) a fiber of~$\Psi_{\G_{12}}$ over a general point of~$\M{2\{1,2\}}$
and in this way we obtain the desired description of~$\config{\G_{12}}{\Lambda_{12}}$.

\begin{proposition}
  \label{proposition:description}
  For a general $\Lambda_{12} \in \M{2\{1,2\}}$, the variety~$\config{\G_{12}}{\Lambda_{12}}$ is isomorphic to the blowup of $\p^1 \times \p^1$ at four points.
\end{proposition}

First of all, we recall the construction of the Segre cubic~$S_3$ as a moduli space of $6$-tuples of points in~$\p^1$ from \cite{HowardMillsonEtAl2009}.
The construction proceeds as follows:
\begin{itemize}
  \item We consider all possible graphs with vertex set $\{P_0, P_1, P_2, Q_0, Q_1, Q_2\}$ where all the vertices have valency one and such that,
  once we draw the vertices as vertices of a regular hexagon and the edges as straight segments, no two edges intersect.
  We obtain five graphs denoted by $X_0$, $X_1$, $X_2$, $Y_1$, and $Y_2$, depicted in \Cref{figure:kempe_graphs}.
  We select an orientation for each of the edges of these graphs.
  In our case, we orient an edge from its vertex above to its vertex below.
  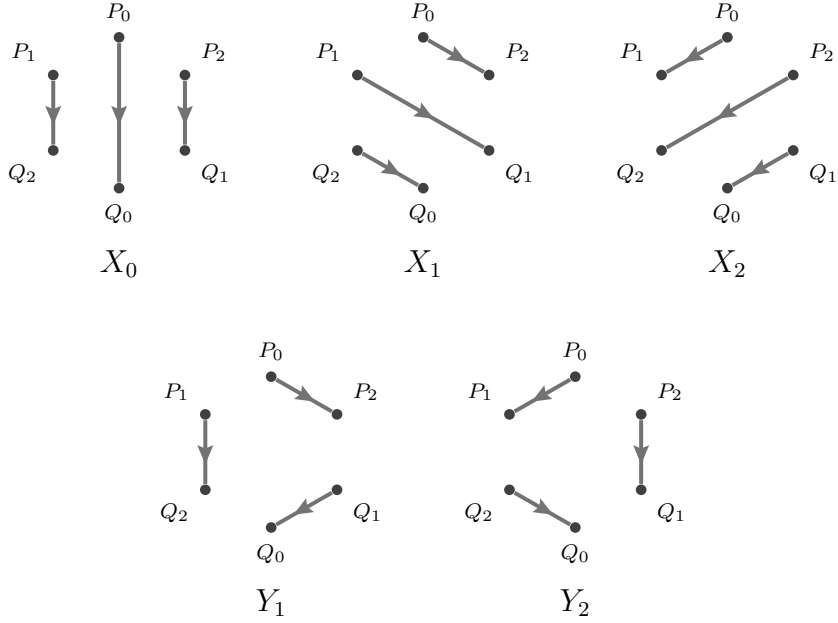
\begin{figure}[ht]
    \centering
    \begin{tikzpicture}
      \begin{scope}
      \hexagon
      \draw[dedge=0.5] (P0)--(Q0);
      \draw[dedge=0.5] (P2)--(Q1);
      \draw[dedge=0.5] (P1)--(Q2);
      \node (L0) at (0,-2) {$X_0$};
      \end{scope}
      \begin{scope}[xshift=4cm]
      \hexagon
      \draw[dedge=0.5] (P1)--(Q1);
      \draw[dedge=0.5] (Q2)--(Q0);
      \draw[dedge=0.5] (P0)--(P2);
      \node (L0) at (0,-2) {$X_1$};
      \end{scope}
    \begin{scope}[xshift=8cm]
      \hexagon
      \draw[dedge=0.5] (P2)--(Q2);
      \draw[dedge=0.5] (Q1)--(Q0);
      \draw[dedge=0.5] (P0)--(P1);
      \node (L0) at (0,-2) {$X_2$};
      \end{scope}
    \begin{scope}[xshift=2cm, yshift=-4.5cm]
      \hexagon
      \draw[dedge=0.5] (P0)--(P2);
      \draw[dedge=0.5] (Q1)--(Q0);
      \draw[dedge=0.5] (P1)--(Q2);
      \node (L0) at (0,-2) {$Y_1$};
      \end{scope}
    \begin{scope}[xshift=6cm, yshift=-4.5cm]
      \hexagon
      \draw[dedge=0.5] (P0)--(P1);
      \draw[dedge=0.5] (P2)--(Q1);
      \draw[dedge=0.5] (Q2)--(Q0);
      \node (L0) at (0,-2) {$Y_2$};
      \end{scope}
    \end{tikzpicture}
    \caption{The five auxiliary graphs that define~$S_3$.}
    \label{figure:kempe_graphs}
  \end{figure}
  \item Each of the vertices of the hexagon corresponds to a copy of~$\p^1$.
  We associate a polynomial to each of the five graphs from \Cref{figure:kempe_graphs} as follows.
  Suppose that $e = \{v, w\}$ is an edge of one of those graphs and suppose that $e$ is oriented from~$v$ to~$w$.
  Let $(s_v: t_v)$ be the coordinates on the~$\p^1$ corresponding to~$v$
  and let~$(s_w: t_w)$ be the coordinates on the~$\p^1$ corresponding to~$w$.
  We set
  \[
    \ell_{e} := s_v t_w - s_w t_v \,.
  \]
  Then for each graph $N = (E_N, V_N)$ in \Cref{figure:kempe_graphs}, we define
  \[
    p_{N} := \prod_{e \in E_N} \ell_e \,.
  \]
  \item We define the rational map $\Gamma \colon (\p^1)^6 \dashrightarrow \p^4$ as $\Gamma := (p_{X_0} : p_{X_1} : p_{X_2} : p_{Y_1} : p_{Y_2})$.
  On $6$-tuples of distinct points, the map~$\Gamma$ behaves as a quotient map,
  namely sends all $\p\mathrm{GL}_2$-equivalent $6$-tuples to the same point.
  \item We consider the subset $\parameter \subset (\p^1)^6$ defined by
  \[
    \parameter :=
    \bigl\{
      (a_1, \dotsc, a_6) \in (\p^1)^6 \,\mid\,
      \text{ at most three } a_i \text{'s coincide}
    \bigr\} \,.
  \]
  The image of~$\parameter$ under~$\Gamma$ is the Segre cubic~$S_3$.
  We write $x_0$, $x_1$, $x_2$, $y_1$, and $y_2$ for the coordinates of~$\p^4$ where $S_3$ lives.
  In these coordinates, the threefold~$S_3$ has the equation
  \[
    x_0 x_1 x_2 - y_1 y_2 (x_0 + x_1 + x_2 - y_1 - y_2) = 0 \,.
  \]
  The real structure in $(\p^1)^6$
  that swaps the coordinates labeled by~$P_i$ with those labeled by $Q_i$ for $i \in \{0,1,2\}$
  induces a real structure on~$S_3 \subset \p^4$ for which the $10$ nodes come into five conjugated pairs.
  This real structure is compatible, via the diagram in \Cref{equation:resolution_segre},
  with the real structure on~$\M{2V_{12}}$ described in \Cref{configurations}.
  Moreover, it determines an automorphism of the coordinate ring of~$\p^4$
  such that the variables $x_i$ are conjugated to $-x_i$ for $i \in \{0,1,2\}$,
  while $y_1$ is conjugated to~$-y_2$.
  The reason for this is that, when in \Cref{figure:kempe_graphs} we swap the vertices labeled with~$P_i$ by those labeled with~$Q_i$, what happens is that each directed graph~$X_i$ is mapped to the same graph with all the edge orientations reversed, while the directed graph~$Y_1$ is mapped to~$Y_2$ with all the edge orientations reversed.
  Due to the rule associating a directed graph to its polynomial explained above,
  changing the orientation of three edges changes the sign of the polynomial.
  This concludes the explanation of the automorphism provided by the real structure.
\end{itemize}

Notice that, from the construction of~$S_3$ that we reported above, it follows that we can define a rational map $S_3 \dashrightarrow \M{2\{1,2\}}$ by specifying a rational map $\parameter \dashrightarrow \p^1$ that is $\p\mathrm{GL}_2$-invariant.
In particular, by inspecting \Cref{equation:resolution_segre}, the map $\Psi_{12} \colon S_3 \dashrightarrow \M{2\{1,2\}}$ is the one induced by the rational map $\parameter \dashrightarrow \p^1$ that takes the cross-ratio of the points~$P_1$, $Q_1$, $P_2$, and~$Q_2$.
In fact, this is precisely what $\Psi_{12}$ does when we apply it to irreducible rational stable curves for which at most three of the marked points $P_1$, $Q_1$, $P_2$, and~$Q_2$ coincide.
Let $(s_1^P: t_1^P)$, $(s_1^Q: t_1^Q)$, $(s_2^P: t_2^P)$, and $(s_2^Q: t_2^Q)$ be their coordinates, respectively.
Their cross-ratio is:
\begin{equation}
  \label{equation:cross_ratio}
  \bigl(
    (s_2^P t_1^Q - s_1^Q t_2^P)(s_2^Q t_1^P - s_1^P t_2^Q)
    :
    (s_2^P t_1^P - s_1^P t_2^P)(s_2^Q t_1^Q - s_1^Q t_2^Q)
  \bigr) \,.
\end{equation}
Our goal is now to infer from \Cref{equation:cross_ratio} an explicit expression of~$\Psi_{\G_{12}}$ as a map from~$S_3$ to~$\p^1$.
If we multiply\footnote{This may enlarge the subset where the rational map is not defined,
but we can recover the missing points
by expressing the rational map in a different way.}
both the homogeneous coordinates in \Cref{equation:cross_ratio} by $(s_0^P t_0^Q - s_0^Q t_0^P)$,
the first component of $\Psi_{\G_{12}}$ equals the polynomial~$p_{X_0}$ from the construction of~$S_3$.
This means that, when we pass to the map~$\Psi_{\G_{12}}$, its first coordinate is~$x_0$.
The second coordinate of the cross ratio from \Cref{equation:cross_ratio} becomes
\begin{equation}
  \label{equation:second_coordinate}
  (s_0^P t_0^Q - s_0^Q t_0^P)(s_2^P t_1^P - s_1^P t_2^P)(s_2^Q t_1^Q - s_1^Q t_2^Q) \,,
\end{equation}
which is the polynomial corresponding to the directed graph in \Cref{figure:graph_second_coordinate},
according to the rule that we used in the construction of~$S_3$.
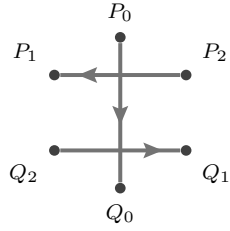
\begin{figure}[ht]
  \centering
  \begin{tikzpicture}
    \hexagon
    \draw[dedge=0.5] (P0)--(Q0);
    \draw[dedge=0.75] (P2)--(P1);
    \draw[dedge=0.75] (Q2)--(Q1);
  \end{tikzpicture}
  \caption{Graph for the polynomial from \Cref{equation:second_coordinate}.}
  \label{figure:graph_second_coordinate}
\end{figure}
Now we apply a result from~\cite[Section~2.2]{HowardMillsonEtAl2009},
the so-called \emph{straightening rules},
to replace the second component of $\Psi_{\G_{12}} \circ \Gamma$ with another one without changing the map.
This is obtained by decomposing the graph from \Cref{figure:graph_second_coordinate}
into ones whose edges do not intersect, as described in \Cref{figure:straightening}.
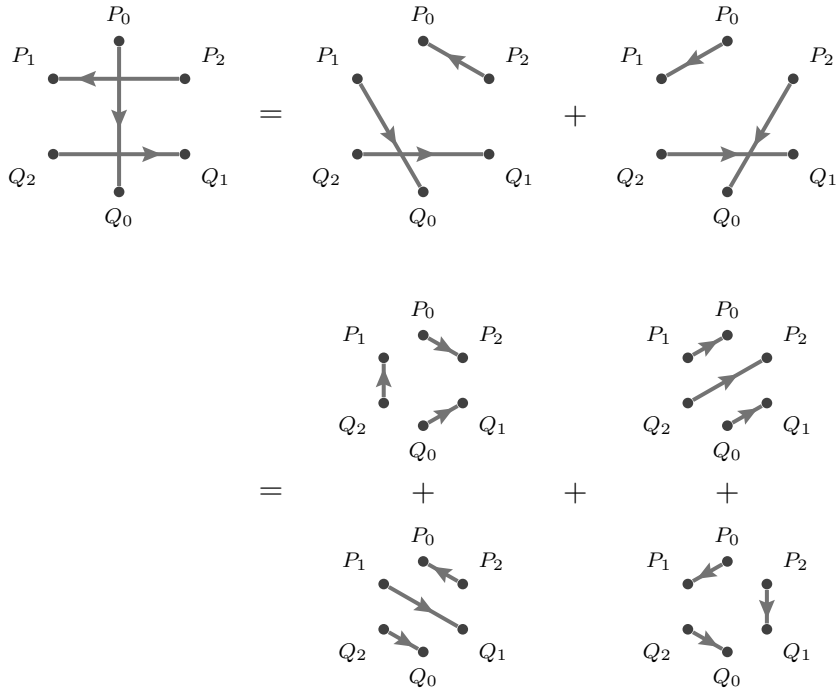
\begin{figure}[ht]
  \centering
  \begin{tikzpicture}
    \begin{scope}
      \hexagon
      \draw[dedge=0.5] (P0)--(Q0);
      \draw[dedge=0.75] (P2)--(P1);
      \draw[dedge=0.75] (Q2)--(Q1);
    \end{scope}
    \begin{scope}[xshift=2cm]
      \node at (0,0) {=};
    \end{scope}
    \begin{scope}[xshift=4cm]
      \hexagon
      \draw[dedge=0.5] (P2)--(P0);
      \draw[dedge=0.5] (P1)--(Q0);
      \draw[dedge=0.5] (Q2)--(Q1);
    \end{scope}
    \begin{scope}[xshift=6cm]
      \node at (0,0) {+};
    \end{scope}
    \begin{scope}[xshift=8cm]
      \hexagon
      \draw[dedge=0.5] (P0)--(P1);
      \draw[dedge=0.5] (P2)--(Q0);
      \draw[dedge=0.5] (Q2)--(Q1);
    \end{scope}
    \begin{scope}[xshift=2cm, yshift=-5cm]
      \node at (0,0) {=};
    \end{scope}
    \begin{scope}[xshift=4cm, yshift=-3.5cm, scale=0.6]
      \hexagon
      \draw[dedge=0.5] (P0)--(P2);
      \draw[dedge=0.5] (Q0)--(Q1);
      \draw[dedge=0.5] (Q2)--(P1);
    \end{scope}
    \begin{scope}[xshift=4cm, yshift=-5cm, scale=0.6]
      \node at (0,0) {+};
    \end{scope}
    \begin{scope}[xshift=4cm, yshift=-6.5cm, scale=0.6]
      \hexagon
      \draw[dedge=0.5] (P2)--(P0);
      \draw[dedge=0.5] (P1)--(Q1);
      \draw[dedge=0.5] (Q2)--(Q0);
    \end{scope}
    \begin{scope}[xshift=6cm, yshift=-5cm]
      \node at (0,0) {+};
    \end{scope}
    \begin{scope}[xshift=8cm, yshift=-3.5cm, scale=0.6]
      \hexagon
      \draw[dedge=0.5] (P1)--(P0);
      \draw[dedge=0.5] (Q0)--(Q1);
      \draw[dedge=0.5] (Q2)--(P2);
    \end{scope}
    \begin{scope}[xshift=8cm, yshift=-5cm, scale=0.6]
      \node at (0,0) {+};
    \end{scope}
    \begin{scope}[xshift=8cm, yshift=-6.5cm, scale=0.6]
      \hexagon
      \draw[dedge=0.5] (P0)--(P1);
      \draw[dedge=0.5] (P2)--(Q1);
      \draw[dedge=0.5] (Q2)--(Q0);
    \end{scope}
  \end{tikzpicture}
  \caption{Straightening procedure for the graph from \Cref{figure:graph_second_coordinate}.}
  \label{figure:straightening}
\end{figure}
This means that the polynomial from \Cref{equation:second_coordinate}
determines the same form on $S_3$ as $-p_{X_1} - p_{X_2} + p_{Y_1} + p_{Y_2}$
(here we also used that reversing the orientation of an edge in a graph changes the sign of the corresponding polynomial).
In other words, the second coordinate of the map~$\Psi_{\G_{12}}$ is given by the linear form
\[
  -x_1 - x_2 + y_1 + y_2 \,.
\]
Thus, the map
\[
  \Psi_{\G_{12}} \colon S_3 \dashrightarrow \M{2\{1,2\}} \cong \p^1 \,,
\]
when $x_0 \neq 0$ or $-x_1 - x_2 + y_1 + y_2 \neq 0$, is:
\begin{equation}
  \label{equation:map_psi}
  (x_0: x_1: x_2: y_1: y_2) \mapsto (x_0: -x_1 - x_2 + y_1 + y_2) \,.
\end{equation}
This map is not defined on the plane $x_0 = -x_1 - x_2 + y_1 + y_2 = 0$.
On the other hand, notice that the equation of~$S_3$ can be written as
\[
  \mathrm{det}
  \left(
  \begin{array}{cc}
    x_0 & -x_1 - x_2 + y_1 + y_2 \\
    -y_1 y_2 & x_1 x_2 - y_1 y_2
  \end{array}
  \right) = 0 \,.
\]
Therefore, in a suitable open subset of~$S_3$, the map from \Cref{equation:map_psi} coincides with
\[
  (x_0: x_1: x_2: y_1: y_2) \mapsto (-y_1 y_2: x_1 x_2 - y_1 y_2) \,.
\]
In this way, we conclude that the map~$\Psi_{\G_{12}}$ is defined on the whole~$S_3$, except for the four nodes
\begin{equation}
  \label{equation:nodes}
  (0:0:1:0:1) \,, \;\;
  (0:0:1:1:0) \,, \;\;
  (0:1:0:1:0) \,, \;\;
  (0:1:0:0:1) \,.
\end{equation}
Notice that the first two and the last two nodes in \Cref{equation:nodes} are conjugated under the real structure of~$S_3$ described before.
It is not possible to further extend the map~$\Psi_{\G_{12}}$ to any of the four previous nodes,
since they are the images of the loci in $\parameter \subset (\p^1)^6$ where three out of the four points~$P_1$, $Q_1$, $P_2$, and~$Q_2$ coincide,
and the cross-ratio is not defined for those four-tuples.

The fiber of~$\Psi_{\G_{12}}$ over a general point $(\mu_1 : \mu_2) \in \p^1$ is the subvariety
\begin{equation}
  \label{equation:fiber}
  F_{\mu_1, \mu_2} :=
  \bigl\{
    \mu_2 x_0 - \mu_1 (-x_1 - x_2 + y_1 + y_2) = \mu_1 x_1 x_2 - (\mu_1 - \mu_2) y_1 y_2 = 0
  \bigr\}
\end{equation}
minus the four nodes from \Cref{equation:nodes}.
For a general $(\mu_1: \mu_2)$, the subvariety~$F_{\mu_1, \mu_2}$ is a smooth two-dimensional quadric,
hence it is isomorphic to~$\p^1 \times \p^1$.
Recall now from \Cref{equation:resolution_segre} that blowing up~$S_3$ at its nodes resolves the indeterminacies of~$\Psi_{\G_{12}}$.
Therefore, a general fiber of the morphism~$\Phi_{\G_{12}}$ is isomorphic to the blowup of~$\p^1 \times \p^1$ at four points.
This proves \Cref{proposition:description}.
When this fiber is taken over a real point, the variety $\config{\G_{12}}{\Lambda_{12}}$ is endowed with a real structure that swaps the two pairs of exceptional divisors of the blowup and the two families of lines in~$\p^1 \times \p^1$.

\section{Classes for spherical calligraphs}
\label{classes}

We now combine the results of \Cref{reduction} and \Cref{description} to define $S^2$-classes of calligraphs and their product.
Moreover, we prove that,
when we have a calligraphic split of a minimally rigid graph,
the number of complex realizations of such a graph on the sphere up to rotations,
is the product of the $S^2$-classes of the two calligraphs in the split.

From \Cref{reduction}, we know that the degree of the class of the configuration space of a graph~$\G$ with a split $\G = \G_1 \cup \G_2$ is
\[
  \int_{\M{2V_{12}}} \chow{\G_{12}} \cdot (\eta_1)_{\ast} (\redchow{\G_1}) \cdot (\eta_2)_{\ast} (\redchow{\G_2}) \,.
\]
From \Cref{description}, we know that $\config{\G_{12}}{\Lambda_{12}}$ is a nonsingular hypersurface in~$\M{2V_{12}}$ for a general~$\Lambda_{12} \in \M{2\{1,2\}}$.
Therefore, the inclusion $i \colon \config{\G_{12}}{\Lambda_{12}} \longrightarrow \M{2V_{12}}$ is a regular embedding,
so by \cite[Example~8.1.1]{Fulton1998}:
\[
  \chow{\G_{12}} \cdot (\eta_1)_{\ast} (\redchow{\G_1}) \cdot (\eta_2)_{\ast} (\redchow{\G_2})
  =
  i_{\ast} i^{\ast} \bigl( (\eta_1)_{\ast} (\redchow{\G_1}) \cdot (\eta_2)_{\ast} (\redchow{\G_2}) \bigr) \,.
\]
Since both $\config{\G_{12}}{\Lambda_{12}}$ and~$\M{2V_{12}}$ are nonsingular,
the map~$i^{\ast}$ is a ring homomorphism between the respective Chow rings, hence
\[
  \int_{\M{2V_{12}}} \chow{\G_{12}} \cdot (\eta_1)_{\ast} (\redchow{\G_1}) \cdot (\eta_2)_{\ast} (\redchow{\G_2})
  =
  \int_{\config{\G_{12}}{\Lambda_{12}}} i^{\ast} (\eta_1)_{\ast} (\redchow{\G_1}) \cdot i^{\ast} (\eta_2)_{\ast} (\redchow{\G_2}) \,.
\]
Summing up, so far we got that
\begin{equation}
  \label{equation:degree_product}
  \int_{\M{2V}} \chow{\G} =
  \int_{\config{\G_{12}}{\Lambda_{12}}} i^{\ast} (\eta_1)_{\ast} (\redchow{\G_1}) \cdot i^{\ast} (\eta_2)_{\ast} (\redchow{\G_2}) \,.
\end{equation}
and to obtain this result we have not used yet that any of~$\G_1$ and~$\G_2$ are calligraphs.
Now comes the time for using this hypothesis.

Let $\calli = (V, E)$ be a calligraph.
If $\Lambda$ is a general choice of edge lengths for~$\calli$,
then $\redconfig{\calli}{\widetilde{\Lambda}}$ from \Cref{definition:reduced_configuration_space} is a smooth surface.
Indeed, the variety $\redconfig{\calli}{\widetilde{\Lambda}}$ is a fiber over a general point of a morphism between smooth manifolds,
so it is smooth by Sard's theorem.
Moreover, since by definition adding $\{0,1\}$ or $\{0,2\}$ makes $\calli$ rigid,
the image of $\redconfig{\calli}{\widetilde{\Lambda}}$ under the forgetful morphism $\eta \colon \M{2V} \longrightarrow \M{2\{0,1,2\}}$ is still a surface.
By construction, this surface is different from~$\config{\G_{12}}{\Lambda_{12}}$, therefore
\[
  i^{\ast} \eta_{\ast} (\redchow{\calli}) \in A^1(\config{\G_{12}}{\Lambda_{12}}) \,,
\]
where $A^1(\cdot)$ are the classes of codimension~$1$ cycles in the Chow ring.
From \Cref{description}, we know that $\config{\G_{12}}{\Lambda_{12}}$ is isomorphic to the blowup of~$\p^1 \times \p^1$ at four points, thus
\begin{equation}
  \label{equation:chow_ring}
  A(\config{\G_{12}}{\Lambda_{12}}) \cong \frac{\Z[F_1, F_2, E_1, E_2, E_3, E_4]}{(F_1^2, F_2^2, F_1 F_2 - 1, \{ F_i E_j \}_{\substack{i \in \{1,2\} \\ j \in \{ 1,\dotsc,4\}}}, \{ E_j^2 + 1 \}_{j \in \{ 1,\dotsc,4\}})} \,.
\end{equation}
In \Cref{equation:chow_ring}, the isomorphism is the one induced by the structure of~$\config{\G_{12}}{\Lambda_{12}}$ as a blowup of~$\p^1 \times \p^1$ described in \Cref{description}, so the classes of the two families of lines of~$\p^1 \times \p^1$ are denoted by~$F_1$ and~$F_2$,
while the classes of the exceptional divisors are denoted by $E_1, \dotsc, E_4$.
The real structure on $\config{\G_{12}}{\Lambda_{12}}$ described in \Cref{description} determines an involution of $A(\config{\G_{12}}{\Lambda_{12}})$ that swaps~$F_1$ with~$F_2$, $E_1$ with~$E_2$, and $E_3$ with~$E_4$.
Taking into account that $\redconfig{\calli}{\widetilde{\Lambda}}$ and $\config{\G_{12}}{\Lambda_{12}}$ are real varieties
when $\Lambda$ comes from a real assignment of edge lengths,
it follows that the quantity $i^{\ast} \eta_{\ast} (\redchow{\calli})$ can be uniquely expressed as
\begin{equation}
  \label{equation:blowup_class}
  a (F_1 + F_2) - b (E_1 + E_2) - c (E_3 + E_4) \,.
\end{equation}
where $a, b, c \in \Z$.

\begin{definition}
  \label{definition:S2_class}
  Let $\calli$ be a calligraph.
  The \emph{$S^2$-class} of~$\calli$, denoted~$[\calli]$, is the triple $(a,b,c) \in \Z^3$ from \Cref{equation:blowup_class} that expresses~$i^{\ast} \eta_{\ast} (\redchow{\calli})$ with respect to the generators of~$A(\config{\G_{12}}{\Lambda_{12}})$.
\end{definition}

\begin{definition}
  \label{definition:quadratic_form}
  We define a quadratic form
  \[
    \begin{array}{rccc}
      \_ \cdot \_ \colon & \Z^3 \times \Z^3 & \longrightarrow & \Z \\
      & (a_1, b_1, c_1), (a_2, b_2, c_3) & \mapsto & 2 (a_1 a_2  - b_1 b_2 - c_1 c_2)
    \end{array} \,.
  \]
  This quadratic form allows us to multiply two $S^2$-classes.
\end{definition}

Now we have the tools to prove the first part of \Cref{theorem:main}.

\begin{proposition}
\label{proposition:main_classes}
  If $\G$ is a minimally rigid graph and $(\G_1, \G_2$) is a calligraphic split of~$\G$,
  then the number of complex realizations of~$\G$ on the sphere up to rotations is equal to the product of
  the classes of $\G_1$ and $\G_2$. In other words, $\scount{\G}=[\G_1] \cdot [\G_2]$.
\end{proposition}
\begin{proof}
  If $\G$ is minimally rigid and $(\G_1, \G_2$) is a calligraphic split of~$\G$,
  then by what we have proven so far and recalling \Cref{proposition:count_degree}, we have
  \[
  \scount{\G}
  =
  \int_{\config{\G_{12}}{\Lambda_{12}}} i^{\ast} (\eta_1)_{\ast} (\redchow{\G_1}) \cdot i^{\ast} (\eta_2)_{\ast} (\redchow{\G_2}) \,.
  \]
  Since both $\G_1$ and $\G_2$ are calligraphs, they have $S^2$-classes $[\G_i] = (a_i, b_i, c_i)$.
  The intersection theory of the blowup of $\p^1 \times \p^1$ at four points shows then that
  \begin{multline*}
    \int_{\config{\G_{12}}{\Lambda_{12}}} i^{\ast} (\eta_1)_{\ast} (\redchow{\G_1}) \cdot i^{\ast} (\eta_2)_{\ast} (\redchow{\G_2}) = \\
    \bigl( a_1(F_1 + F_2) - b_1(E_1 + E_2) - c_1(E_3 + E_4) \bigr)\!
    \bigl( a_2(F_1 + F_2) - b_2(E_1 + E_2) - c_2(E_3 + E_4) \bigr) = \\
    2 (a_1 a_2  - b_1 b_2 - c_1 c_2) \,,
  \end{multline*}
  which, together with the previous discussion, concludes the proof.
\end{proof}

The second part of \Cref{theorem:main} follows from a concrete computation.

Let us compute the $S^2$-class of the calligraph~$\lef$.
To this end, we try to understand $\redconfig{\lef}{\widetilde{\Lambda}}$ for a general
\[
  \Lambda = (\Lambda_{01}, \Lambda_{12}) \in \M{2\{0,1\}} \times \M{2\{1,2\}}
\]
and its intersection with $\config{\G_{12}}{\Lambda_{12}}$,
which is equal to~$\config{\lef}{\Lambda}$.
As in \Cref{description}, we examine the corresponding subvarieties in the Segre cubic~$S_3$.
Using the same arguments that yield \Cref{equation:fiber}, we express the blowdown of~$\config{\lef}{\Lambda}$ in~$S_3$ as the intersection of the two loci:
\begin{align*}
  \mu_2 x_0 - \mu_1 (x_1 + x_2 - y_1- y_2) = \mu_1 x_1 x_2 - (\mu_1 + \mu_2) y_1 y_2 &= 0 \quad \text{and} \\
  \delta_2 x_2 - \delta_1 (x_0 + x_1 - y_1 - y_2) = \delta_1 x_0 x_1 - (\delta_1 + \delta_2) y_1 y_2 &= 0 \,,
\end{align*}
where $\Lambda_{12} = (\mu_1: \mu_2)$ and $\Lambda_{01} = (\delta_1: \delta_2)$,
once we identify both~$\M{2\{1,2\}}$ and~$\M{2\{0,1\}}$ with~$\p^1$.
A computer algebra calculation shows that this intersection has equations
\begin{align*}
  \mu_1 (\delta_1 \mu_1 - \delta_2 \mu_2) x_2^2 - \delta_1 \mu_1 (\mu_1 + \mu_2) x_2 y_1 - \delta_1 \mu_1 (\mu_1 + \mu_2) x_2 y_2 +  \delta_1 (\mu_1 + \mu_2)^2 y_1 y_2 &= 0 \,, \\
  \delta_1 (\mu_1 + \mu_2) x_0 - (\mu_1 \delta_1 + \mu_2 \delta_2) x_2 &= 0 \,, \\
  \delta_1 (\mu_1 + \mu_2) (x_1 - y_1 - y_2) + (\mu_1 \delta_1 - \mu_2 \delta_2) x_2 &= 0 \,,
\end{align*}
namely, it is a conic passing through the two nodes
\[
  (0: 1: 0: 1: 0)
  \quad \text{and} \quad
  (0: 1: 0: 0: 1)
\]
(which are conjugated under the real structure of~$S_3$)
and no other node of~$S_3$.
Therefore, the $S^2$-class of~$\lef$ is $(1, 1, 0)$.
By symmetry, the $S^2$-class of~$\righ$ is $(1, 0, 1)$.

Now that we know the $S^2$-classes of $\lef$ and $\righ$,
we can constrain the possibilities for $[\cente{v}]$ from the knowledge of the number of realizations of some simple graphs.
In fact, both $\lef \cup \cente{v}$ and $\righ \cup \cente{v}$ are isomorphic to the graph
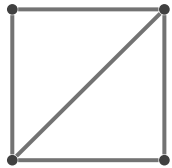
\begin{center}
  \begin{tikzpicture}[scale = 2]
    \node[vertex] (A) at (0,0) {};
    \node[vertex] (B) at (0,1) {};
    \node[vertex] (C) at (1,0) {};
    \node[vertex] (D) at (1,1) {};
    \draw[edge] (A)edge(B) (B)edge(D) (D)edge(C) (C)edge(A) (A)edge(D);
  \end{tikzpicture}
\end{center}
whose number of realizations on the sphere, up to rotations, is $4$.
This means that $[\lef] \cdot [\cente{v}] = [\righ] \cdot [\cente{v}] = 4$.
Hence, if we write $[\cente{v}] = (a, b, c)$, then
\[
  \left\{
  \begin{array}{l}
    2(a - b) = 4 \,, \\
    2(a - c) = 4 \,,
  \end{array}
  \right.
  \quad \Rightarrow \quad
  \left\{
  \begin{array}{l}
    b = c \,, \\
    a = b + 2 \,.
  \end{array}
  \right.
\]
Since $\cente{v} \cup \cente{w}$ is isomorphic to
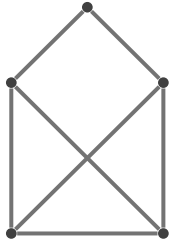
\begin{center}
  \begin{tikzpicture}[scale = 2]
    \node[vertex] (A) at (0,0) {};
    \node[vertex] (B) at (0,1) {};
    \node[vertex] (C) at (1,0) {};
    \node[vertex] (D) at (1,1) {};
    \node[vertex] (E) at (0.5, 1.5) {};
    \draw[edge] (A)edge(B) (B)edge(C) (C)edge(D) (C)edge(A) (A)edge(D) (B)edge(E) (D)edge(E);
  \end{tikzpicture}
\end{center}
we know that $[\cente{v}] \cdot [\cente{w}] = 8$, so $a^2 - b^2 - c^2 = 4$.
This, together with the previous two relations, implies that there are only two possibilities for~$[\cente{v}]$, namely
\[
  (2, 0, 0)
  \quad \text{or} \quad
  (6, 4, 4) \,.
\]
To prove that $[\cente{v}]$ is precisely~$(2,0,0)$, we employ a technique that relies on a deeper understanding of the workings of the divisors~$D_{I,J}$.
Therefore, in order not to weigh down the exposition, we decided to place the final part of this proof in an appendix (see \Cref{sec:s2class}).

This last discussion, together with \Cref{proposition:main_classes}, concludes the proof of \Cref{theorem:main}.

We end the section by showing that, for our purposes, it is actually irrelevant whether $[\cente{v}] = (2,0,0)$ or $[\cente{v}] = (6,4,4)$; see \Cref{remark:644}.
In fact, there is an automorphism of $\Z$-modules of~$\Z^3$ into itself that maps $(1,1,0)$ and~$(1,0,1)$ to themselves and sends $(2,0,0)$ to $(6,4,4)$.
This map is given by the matrix
\[
  O =
  \begin{pmatrix}
    3 & -2 & -2\\
    2 & -1 & -2\\
    2 & -2 & -1
  \end{pmatrix} \,.
\]
One can then check that if $A$ is the matrix of the quadratic form in \Cref{theorem:main}, namely
\[
  A =
  \begin{pmatrix}
    2 & 0 & 0 \\
    0 & -2 & 0\\
    0 & 0 & -2
  \end{pmatrix} \,,
\]
then $O A O^{t} = A$, i.\,e., the transformation $O$ preserves the quadratic form of~$A$.
Hence, the matrix~$O$ determines an automorphism of the subring of real classes in the Chow ring of $\config{\G_{12}}{\Lambda_{12}}$
that keeps the $S^2$-classes of~$\lef$ and~$\righ$ invariant, and swaps $(2,0,0)$ and $(6,4,4)$.
Therefore, any of these two triples could be chosen as the $S^2$-class of~$\cente{v}$.

\begin{remark}
  Define the \emph{coupler curve} and \emph{coupler multiplicity} for calligraphs as in
  \cite[Section~2.2]{GraseggerElHilanyLubbes2024}, but for the sphere instead of the plane.
  If $(2,0,0)$ is chosen as the $S^2$-class of $\cente{v}$, then the first integer of the $S^2$-class
  of a calligraph equals (up to coupler multiplicity)
  the degree of a general coupler curve~$D \subset S^2_{\C}$ of this calligraph.
  The second number corresponds to the multiplicity of~$D$ at those complex conjugate points
  on the conic at infinity, such that the planes passing through these points are parallel to
  the tangent plane of~$S^2_{\C}$ at the realization of vertex~$1$.
  Similarly for the third number, but with vertex~$1$ replaced by vertex~$2$.
\end{remark}

\section{The recursion}
\label{recursion}

In this section, we use the theory we have been developing so far, in particular \Cref{theorem:main},
to synthesize a recursive procedure for the computation of the number of realizations on the sphere,
up to rotations, of a minimally rigid graph.
This can be considered as an ``enhanced'' version of the algorithm developed in~\cite{GalletGraseggerSchicho2020}.

Say that we are given a minimally rigid graph~$\G$ and we want to compute~$\scount{\G}$.
If $(\G_1, \G_2)$ is a calligraphic split of~$\G$, then by \Cref{theorem:main} we know that
\[
  \scount{\G}
  =
  (a_1, b_1, c_1) \cdot (a_2, b_2, c_2) \,,
\]
where $(a_i, b_i, c_i)$ is the $S^2$-class of~$\G_i$.
To compute these two $S^2$-classes we employ the ``basic calligraphs'' from \Cref{definition:basic_calligraphs},
whose $S^2$-classes are already known from \Cref{theorem:main}.

Let us write $\calli$ instead of~$\G_1$ to refresh the notation (the same procedure shall be applied to~$\G_2$).
If $\calli = (V, E)$, we construct the following three graphs:
\[
  \calli_\lef := (V, E \cup E_\lef) \,, \quad
  \calli_\righ := (V, E \cup E_\righ) \,, \quad
  \calli_{\cente{v}} := (V \cup \{v\}, E \cup E_{\cente{v}}) \,,
\]
where $v$ is a vertex not in~$\calli$ and $E_\lef$, $E_\righ$, and $E_{\cente{v}}$ are the edges of~$\lef$, $\righ$, and~$\cente{v}$, respectively.
Notice that $(\calli, \lef)$, $(\calli, \righ)$, and $(\calli, \cente{v})$ are calligraphic splits of~$\calli_\lef$, $\calli_\righ$, and $\calli_{\cente{v}}$, respectively.

Suppose that all $\calli_\lef$, $\calli_\righ$, and $\calli_{\cente{v}}$ are minimally rigid.
Let us denote $[\calli] = (a,b,c)$.
\Cref{theorem:main} implies that
\begin{equation}
  \label{equation:system}
  \begin{aligned}
    (a,b,c) \cdot (1,1,0) &= 2(a-b) = \scount{\calli_\lef} \,, \\
    (a,b,c) \cdot (1,0,1) &= 2(a-c) = \scount{\calli_\righ} \,, \\
    (a,b,c) \cdot (2,0,0) &= 4a = \scount{\calli_{\cente{v}}} \,.
  \end{aligned}
\end{equation}
If the right hand sides of \Cref{equation:system} are known,
then we can solve for $(a,b,c)$ and obtain the $S^2$-class of~$\calli$.
To get the values~$\scount{\calli_\lef}$, $\scount{\calli_\righ}$, and~$\scount{\calli_{\cente{v}}}$
we can then reiterate the procedure by finding calligraphic splits of each of these graphs.

Here there is the possibility that one of $\calli_\lef$, $\calli_\righ$, or~$\calli_{\cente{v}}$ equals the initial minimally rigid graph~$\G$.
This would start an infinite recursion,
which hence would not result in an algorithm for the computation of $\scount{\G}$.
To prevent this from happening, we ask that the initial calligraphic split $(\G_1, \G_2)$ is \emph{non-trivial}, namely that both~$\G_1$ and~$\G_2$ have at least $5$ vertices.
In this way, we ensure that the iteration of the procedure does not loop back to the initial situation.

It may happen, however, that one of $\calli_\lef$, $\calli_\righ$, and $\calli_{\cente{v}}$, for example $\calli_\lef$, is not minimally rigid.
Then, since $\calli$ is a calligraph, this means that $\calli_\lef$ has still one degree of freedom or its configuration space is empty.
In these cases the degree of~$\chow{\calli_\lef}$ is zero, so \Cref{equation:degree_product} implies that $(a, b, c) \cdot (1,1,0) = 0$,
namely $a = b$.
We get the equations $a = c$, respectively $a = 0$, if $\calli_{\righ}$, respectively $\calli_{\cente{v}}$, is not minimally rigid.

Therefore, whether the graphs $\calli_\lef$, $\calli_\righ$, and $\calli_{\cente{v}}$ are minimally rigid or not,
they can be used to determine the $S^2$-class of~$\calli$.
Notice that, in order for this recursive procedure to become an algorithm,
it is crucial that the graphs $\calli_\lef$, $\calli_\righ$, and $\calli_{\cente{v}}$,
when minimally rigid, are smaller than the graph~$\G$ we started with.
However, it is not always possible to find a calligraphic split of~$\G$ that has this property.
In this case, we fall back and use the algorithm from \cite{GalletGraseggerSchicho2020} to compute the number of realizations.

We summarize in \Cref{algorithm:calligraphic_split} the procedure we described in this section.

\renewcommand{\thealgorithm}{\texttt{CountByCalligraphicSplit}}
\begin{algorithm}[ht]
  \caption{}\label{algorithm:calligraphic_split}
  \begin{algorithmic}[1]
    \Require A minimally rigid graph~$\G$.
    \Ensure The number $\scount{\G}$ of realizations of~$\G$ on the complex sphere, up to~$\mathrm{SO}_3(\C)$.
    \Statex
    \If{there is no non-trivial calligraphic split of $\G$}
      \State \Return $\scount{\G}$ using the algorithm from \cite{GalletGraseggerSchicho2020}.
    \EndIf
    \State {\bfseries Construct} a non-trivial calligraphic split $(\G_1, \G_2)$ of $\G$.
    \State \Return \quadform{\class{$\G_1$}}{\class{$\G_2$}}.
  \end{algorithmic}
\end{algorithm}

\renewcommand{\thealgorithm}{\texttt{QuadForm}}
\begin{algorithm}[ht]
  \caption{}\label{algorithm:quad_form}
  \begin{algorithmic}[1]
    \Require Two integer vectors $(a_1, b_1, c_1)$ and $(a_2, b_2, c_2)$.
    \Ensure Their product according to the quadric form from \Cref{theorem:main}.
    \Statex
    \State \Return $2(a_1 a_2 - b_1 b_2 - c_1 c_2)$.
  \end{algorithmic}
\end{algorithm}

\renewcommand{\thealgorithm}{\texttt{S$^2$-Class}}
\begin{algorithm}[ht]
  \caption{}\label{algorithm:class}
  \begin{algorithmic}[1]
    \Require A calligraph $\calli$.
    \Ensure The $S^2$-class $[\calli]$.
    \Statex
    \If{$\calli = \lef$}
      \State \Return $(1,1,0)$.
    \ElsIf{$\calli = \righ$}
      \State \Return $(1,0,1)$.
    \ElsIf{$\calli = \cente{v}$}
      \State \Return $(2,0,0)$.
    \EndIf
    \State {\bfseries Initialize} \texttt{eqs} := empty list.
    \For{$F \in \{\lef, \righ, {\cente{v}} \}$}
      \If{$\calli_F$ is minimally rigid}
        \State {\bfseries Compute} $r := \text{CountByCalligraphicSplit}(\calli_F)$.
        \State {\bfseries Append} \quadform{\class{$F$}}{$(a,b,c)$} $= r$ to \texttt{eqs}.
      \Else
        \State {\bfseries Append} \quadform{\class{$F$}}{$(a,b,c)$} $= 0$ to \texttt{eqs}.
      \EndIf
    \EndFor
    \State {\bfseries Solve} the equations in \texttt{eqs} with respect to the variables $a$, $b$, and $c$.
    \State \Return $(a,b,c)$.
  \end{algorithmic}
\end{algorithm}

\begin{example}
  Taking the example from \cite[Appendix B]{GraseggerElHilanyLubbes2024} we have a graph with a calligraphic split (see \Cref{fig:pex}).
  Similarly to the plane the realization count on the sphere cannot be computed by previous algorithms but it can by our recursion.
  The two realization counts in the plane and on the sphere differ due to the leaves in the recursion tree but the algorithm is the very same.
  We get
  \[
    \scount{G}=491520 = (384, 0, 0) \cdot (640, 256, 384).
  \]
  The classes of the component graphs can be computed recursively as described in \Cref{algorithm:class}
  using the calligraphic splits from \cite{GraseggerElHilanyLubbes2024}.
  \begin{figure}[ht]
    \centering
    \begin{tikzpicture}[scale=0.4,rotate=-90]
        \node[mvertex] (0) at (2,8) {};
        \node[anchorvertex] (1) at (0,0) {};
        \node[anchorvertex] (2) at (8,0) {};
        \node[vertexh] (3) at (0,10) {};
        \node[vertexh] (4) at (2,10) {};
        \node[vertexh] (5) at (4,8) {};
        \node[vertexh] (6) at (0,8) {};
        \node[vertexh] (7) at (8,10) {};
        \node[vertexh] (8) at (6,8) {};
        \node[vertexh] (9) at (6,6) {};
        \node[vertex] (10) at (2,4) {};
        \node[vertex] (11) at (5,3) {};
        \node[vertex] (12) at (2,2) {};
        \node[vertex] (13) at (4,4) {};
        \node[vertex] (14) at (5,5) {};
        \node[vertex] (15) at (2,6) {};
        \node[vertex] (16) at (4,6) {};
        \draw[edgeh] (3)edge(4);
        \draw[edgeh] (3)edge(0);
        \draw[edgeh] (3)edge(6);
        \draw[edgeh] (4)edge(6);
        \draw[edgeh] (4)edge(7);
        \draw[edgeh] (5)edge(0);
        \draw[edgeh] (5)edge(7);
        \draw[edgeh] (5)edge(8);
        \draw[edgeh] (6)edge(1);
        \draw[edgeh] (6)edge(0);
        \draw[edgeh] (0)edge(9);
        \draw[edgeh] (7)edge(2);
        \draw[edgeh] (7)edge(8);
        \draw[edgeh] (2)edge(9);
        \draw[edgeh] (2)edge(9);
        \draw[edgeh] (9)edge(8);
        \draw[edge] (0)edge(15);
        \draw[edge] (0)edge(16);
        \draw[edge] (1)edge(12);
        \draw[edge] (1)edge(10);
        \draw[edge] (1)edge(11);
        \draw[edge] (2)edge(12);
        \draw[edge] (2)edge(11);
        \draw[edge] (10)edge(15);
        \draw[edge] (10)edge(13);
        \draw[edge] (11)edge(14);
        \draw[edge] (12)edge(13);
        \draw[edge] (13)edge(14);
        \draw[edge] (14)edge(15);
        \draw[edge] (14)edge(16);
        \draw[edge] (16)edge(15);
        \draw[edge, draw=colanchor] (1)edge(2);
    \end{tikzpicture}
    \caption{A calligraphic split of a graph $G=G_1 \cup G_2$.}
    \label{fig:pex}
  \end{figure}
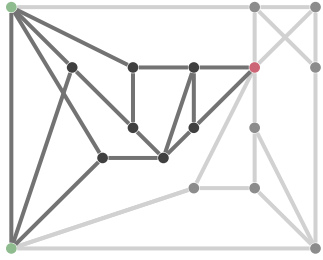
\end{example}

The implementation \cite{GalletGraseggerSchicho2022software} of the algorithm from \cite{GalletGraseggerSchicho2020} for computing realization counts on the sphere is much slower than the planar implementation \cite{CapcoGalletEtAl2018software} for \cite{CapcoGalletEtAl2018}. This is partially for implementation reasons but mainly due to the algorithms themselves.
For this reason the algorithm described in the current paper can take advantage already for smaller graphs.
A timing comparison can be seen in \Cref{fig:timing}.

\begin{figure}[ht]
  \centering
  \begin{tikzpicture}[yscale=1.5,hpercent/.style={black!40!white,dotted,line width=1pt}]
    \coordinate (o) at (7.5,0);
    \begin{scope}[xshift=8cm]
      \fill[colG] (-0.25,0) -- (-0.25,2) -- (-0.1,2.15)--(0.1,1.85) -- (0.25,2) -- (0.25,0) --cycle;
      \draw[hpercent] (-0.3,1) -- (0.3,1);
    \end{scope}
    \foreach \x [count=\i from 9] in {0.841602, 0.308769, 0.314561}
    {
      \fill[colG] (\i-0.25,0) rectangle (\i+0.25,\x);
      \draw[hpercent] (\i-0.3,1) -- (\i+0.3,1);
    }

    \draw[-{Classical TikZ Rightarrow[]}] (o) -- +(0,2.5) node[left,align=left,labelsty] {\%};
    \draw[-{Classical TikZ Rightarrow[]}] (o) -- +(5,0) node[below right,labelsty] {\#vertices};
    \foreach \y [evaluate=\y as \yc using \y/100] in {0,50,100,150,200} \draw ($(o)+(0,\yc)$) -- +(-0.05,0) node[left,labelsty] {$\y$};
    \foreach \y in {0.25,0.75,...,2} \draw (o)++(0,\y) -- +(-0.05,0);
    \foreach \y in {0.5,1,1.5,2} \draw (o)++(0,\y) -- +(-0.1,0);
    \foreach \x in {8,9,...,11} \draw (\x,0) -- +(0,-0.05) node[below,labelsty] {$\x$};
  \end{tikzpicture}
  \caption{The timing for \Cref{algorithm:calligraphic_split} relative to Algorithm~\cite{GalletGraseggerSchicho2020}
    using the set of all minimally rigid graphs that have a non-trivial calligraphic split.
    The $100\%$ line represents the timing for Algorithm~\cite{GalletGraseggerSchicho2020}.
    We see that using the splitting we are already  faster for graphs with at least 9 vertices.
  }
  \label{fig:timing}
\end{figure}
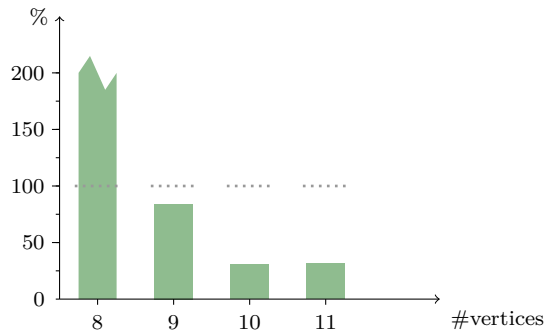

We indicated throughout the paper that the recursive algorithm presented here is indeed the same as the one from \cite{GraseggerElHilanyLubbes2024} just the fallback algorithm is different,
i.\,e. the algorithm that is used when there is no calligraphic split.
With that respect the implementation \cite{GraseggerElHilanyLubbes2022software} in its Version 1 could be immediately used together with \cite{Grasegger2022software} for doing computations on the sphere.
We updated the implementation to make the workflow simpler.

\appendix
\section{The \texorpdfstring{$S^2$}{S2}-class of \texorpdfstring{$\cente{v}$}{Cv}}\label{sec:s2class}

The goal of this appendix is to prove that the $S^2$-class of $\cente{v}$ is $(2,0,0)$ starting from what we have proven in \Cref{classes}.
Since we have already showen that there are only two possibilities for~$[\cente{v}]$, namely $(2,0,0)$ and $(6,4,4)$, it is enough to show that $[\cente{v}]$ is different from~$(6,4,4)$.

First of all, we describe the exceptional locus of the blowup map $\M{2V_{12}} \rightarrow S_3$ defined in \Cref{description}.
\begin{proposition}
  \label{proposition:exceptional_locus}
  The exceptional locus of the blowup map $\M{2V_{12}} \rightarrow S_3$ is the union of the supports of the divisors $D_{I, J}$ such that $\left| I \right| = \left| J \right| = 3$.
\end{proposition}
\begin{proof}
  In this proof we use some concepts from the theory of quotients of algebraic varieties, i.\,e., invariant theory.
  Specifically, we use the notion of \emph{stable} and \emph{semistable} elements; for a reference, see \cite[Section~8.1]{Dolgachev2003}.

  As we clarified in \Cref{description}, the variety $S^3$ is a moduli space of $6$-tuples of points in~$\p^1$ with at most three coinciding points. In particular, there is a rational map $\left( \p^1 \right)^6 \dasharrow S^3$ whose general fibers are the orbits of $6$-tuples of different points under the action of $\p\mathrm{GL}_2$.
  This map factors through the map $\M{2V_{12}} \rightarrow S^3$ that is the blowup at the $10$ nodes of~$S^3$, mentioned at the beginning of \Cref{description}, since $\M{2V_{12}}$ is the moduli space of stable curves with $6$ marked points.
  From \cite[Example~11.6]{Dolgachev2003}, we know that the preimages under $\left( \p^1 \right)^6 \rightarrow S_3$ of the nodes of~$S_3$ are constituted of the elements that are semistable but not stable.
  Now, \cite[Theorem~11.2]{Dolgachev2003} states that the latter elements are the $6$-tuples of points in~$\p^1$ such that there are exactly three points coinciding: in fact, using the reference's notation, we have $n = 1$, $m = 6$, $k_i = 1$ for all $i \in \{1, \dotsc, 6\}$ and $W$ can only be a point of~$\p^1$, so the condition for semistability is that the number of points in a tuple that coincides with~$W$ is less than or equal to the quantity
  \[
    \frac{\dim(W) + 1}{n + 1} \left( \sum_{i = 1}^{m} k_i \right) = \frac{1}{2} \cdot 6 = 3 \,.
  \]
  The image in $\M{2V_{12}}$ of the set of these $6$-tuples is precisely the union of the supports of the divisors described in the statement.
  In fact, as explained in \cite[Definition~9]{GalletGraseggerSchicho2020} where the notation is introduced (slightly modifying the one used in \cite{Keel1992}), the elements of the support of a divisor~$D_{I,J}$ are the stable rational curves that have two components, one with $|I|$ marked points and the other with $|J|$ marked points.
\end{proof}

We now show that $[\cente{v}] = (2,0,0)$.
To do so, it is enough to show that, when $\Lambda \in \M{2\{1,2\}} \times \M{2\{1,v\}} \times \M{2\{2,v\}} \times \M{2\{0,v\}}$ is general and $\widetilde{\Lambda} := \bigl( \Lambda_{1v}, \Lambda_{2v}, \Lambda_{0v} \bigr)$, the curve
\[
  \eta(\redconfig{\cente{v}}{\widetilde{\Lambda}}) \cap \config{\G_{12}}{\Lambda_{12}}
\]
does not intersect any divisor~$D_{I, J}$ of $\M{2V_{12}}$ with $|I| = |J| = 3$, where $\eta \colon \M{2V_{\cente{v}}} \rightarrow \M{2V_{12}}$ is the forgetful morphism.

Indeed, in light of \Cref{proposition:exceptional_locus}, this means that the image of that curve in~$S_3$ does not pass through any of the nodes of the cubic hypersurface~$S_3$, which in turn by \Cref{definition:S2_class} means that the second and third entry of the $S^2$-class $[\cente{v}]$ must be zero, ruling out the possibility $[\cente{v}] = (6, 4, 4)$.
Let us first understand which divisors $D_{I, J} \subset \M{2V_{12}}$ with $|I| = |J| = 3$ intersect $\config{\G_{12}}{\Lambda_{12}}$.
From \cite[Fact~2]{Keel1992}, we know that, for any divisor~$D_{I,J}$ with $|I \cap \{P_1, Q_1, P_2, Q_2\}| = 2$, the edge forgetful morphism $\pi_{\{1,2\}} \colon \M{2V_{12}} \rightarrow \M{2\{1,2\}}$ is constant and equal to one of the three special values $(1:0)$, $(0:1)$, or $(1:1)$.
Notice that the convention we follow is that, for a divisor $D_{I,J}$ in $\M{2V}$, where $G = (V, E)$ is a graph, the sets~$I$ and $J$ are a partition of the set $\{P_w, Q_w\}_{w \in E}$ of marked points, while \cite{Keel1992} denotes these divisors by~$D^T$, where $T$ is a subset of the set~$\{P_w, Q_w\}_{w \in E}$, so one can think that $T$ plays the role of~$I$ and the complement of~$T$ in~$\{P_w, Q_w\}_{w \in E}$ plays the role of~$J$.
However, the variety~$\config{\G_{12}}{\Lambda_{12}}$ is by definition a fiber of~$\pi_{\{1,2\}}$ over a general point of~$\M{2V_{12}}$, hence it cannot intersect divisors $D_{I,J}$ such that $|I \cap \{P_1, Q_1, P_2, Q_2\}| = 2$.
Then, the only divisors~$D_{I,J}$ with $|I| = |J| = 3$ that intersect~$\config{\G_{12}}{\Lambda_{12}}$ are those such that $|I \cap \{P_1, Q_1, P_2, Q_2\}| \not= 2$.
There are only four possibilities (together with the other four conjugate divisors obtained by swapping~$I$ with~$J$), namely:
\[
  I = \{ P_0, Q_0, P_1 \}
  \quad \text{or} \quad
  I = \{ P_0, Q_0, Q_1 \}
  \quad \text{or} \quad
  I = \{ P_0, Q_0, P_2 \}
  \quad \text{or} \quad
  I = \{ P_0, Q_0, Q_2 \}
  \,.
\]

To conclude the computation, notice that, for the curve $\eta(\redconfig{\cente{v}}{\widetilde{\Lambda}}) \cap \config{\G_{12}}{\Lambda_{12}}$ to intersect one of the divisors~$D_{I,J}$ above, it must happen that $\redconfig{\cente{v}}{\widetilde{\Lambda}}$ intersects~$\eta^{-1}(D_{I,J})$.
We focus on~$D_{I,J}$ with $I = \{ P_0, Q_0, P_1 \}$, since the other three cases are completely analogous.
The preimage~$\eta^{-1}(D_{I,J})$ is the union of the supports four divisors in~$\M{2V_{\cente{v}}}$.They are also of the form~$D_{M,N}$, where
\begin{align*}
  M &= \{ P_0, Q_0, P_1, P_v, Q_v \} \,, & N &= \{ Q_1, P_2, Q_2 \} \,, \\
  M &= \{ P_0, Q_0, P_1 \} \,, & N &= \{ Q_1, P_2, Q_2, P_v, Q_v \} \,, \\
  M &= \{ P_0, Q_0, P_1, P_v \} \,, & N &= \{ Q_1, P_2, Q_2, Q_v \} \,, \\
  M &= \{ P_0, Q_0, P_1, Q_v \} \,, & N &= \{ Q_1, P_2, Q_2, P_v \} \,.
\end{align*}
We denote them by~$D_1$, $D_2$, $D_3$, and~$D_4$.
With an argument analogous to the one just used for~$\config{\G_{12}}{\Lambda_{12}}$, the fact that $\Lambda_{0v}$ is general implies that $\redconfig{\cente{v}}{\widetilde{\Lambda}}$ does not intersect~$D_2$ (since $|M \cap \{P_0, Q_0, P_v, Q_v\}| = 2$); the fact that $\Lambda_{1v}$ is general implies that $\redconfig{\cente{v}}{\widetilde{\Lambda}}$ does not intersect~$D_3$ and~$D_4$ (since $|M \cap \{P_1, Q_1, P_v, Q_v\}| = 2$); the fact that $\Lambda_{2v}$ is general implies that $\redconfig{\cente{v}}{\widetilde{\Lambda}}$ does not intersect~$D_1$ (since $|M \cap \{P_2, Q_2, P_v, Q_v\}| = 2$).
This concludes the proof that $[\cente{v}] = (2,0,0)$.

\newcommand{\etalchar}[1]{$^{#1}$}

\bigskip
\bigskip
\textsc{(MG) University of Trieste,
Department of Mathematics, Informatics and Geosciences,
Via Valerio 12/1, 34127 Trieste, Italy}\\
Email address: \texttt{matteo.gallet@units.it}

\textsc{(GG, NL) Johann Radon Institute for Computational and Applied Mathematics (RICAM), Austrian Academy of Sciences}\\
Email address: \texttt{georg.grasegger@ricam.oeaw.ac.at}\\
\phantom{Email address: }\texttt{info@nielslubbes.com}

\textsc{(JS) Johannes Kepler University Linz, Research Institute for Symbolic Computation (RISC)}\\
Email address: \texttt{josef.schicho@risc.jku.at}

\end{document}